\documentclass[12pt]{amsart}
\usepackage{amsmath, amssymb, amsthm, mathrsfs, lscape, graphicx, color, enumerate, textcomp, colonequals, fullpage}
\usepackage[backref]{hyperref}
\usepackage[alphabetic,backrefs,lite]{amsrefs}

\makeatletter
\newcommand{\pushright}[1]{\ifmeasuring@#1\else\omit\hfill$\displaystyle#1$\fi\ignorespaces}
\newcommand{\pushleft}[1]{\ifmeasuring@#1\else\omit$\displaystyle#1$\hfill\fi\ignorespaces}
\makeatother

\newtheorem{thm}[subsection]{Theorem}

\newtheorem*{defn}{Definition}
\newtheorem{corollary}[subsection]{Corollary}
\newtheorem{lemma}[subsection]{Lemma}

\newtheorem{remark}[subsection]{Remark}

\newcommand{\isom}{\cong}

\newcommand{\Q}{\mathbb Q}
\newcommand{\Z}{\mathbb Z}

\renewcommand{\O}{\mathcal{O}}

\renewcommand{\P}{\mathbb P}

\renewcommand{\phi}{\varphi}

\newcommand{\Spec}{\mathrm{Spec}\ }

\DeclareMathOperator{\Art}{Art}
\DeclareMathOperator{\Gal}{Gal}
\DeclareMathOperator{\Aut}{Aut}
\DeclareMathOperator{\wt}{wt}
\DeclareMathOperator{\Pic}{Pic}
\DeclareMathOperator{\disc}{disc}
\DeclareMathOperator{\cha}{char}

\title[Conductors and minimal discriminants]{Conductors and minimal discriminants of hyperelliptic curves with rational Weierstrass points}
\author{Padmavathi Srinivasan}

\begin{document}
 
\date{\today}

\begin{abstract}
 Let $C$ be a hyperelliptic curve of genus $g$ over the fraction field $K$ of a discrete valuation ring $R$. Assume that the residue field $k$ of $R$ is perfect and that $\cha k \neq 2$. Assume that the Weierstrass points of $C$ are $K$-rational. Let $S = \Spec R$. Let $\mathcal{X}$ be the minimal proper regular model of $C$ over $S$. Let $\Art (\mathcal{X}/S)$ denote the Artin conductor of the $S$-scheme $\mathcal{X}$ and let $\nu (\Delta)$ denote the minimal discriminant of $C$. We prove that $-\Art (\mathcal{X}/S) \leq \nu (\Delta)$. As a corollary, we obtain that the number of components of the special fiber of $\mathcal{X}$ is bounded above by $\nu(\Delta)+1$.   
\end{abstract}

\maketitle

\section{Introduction}
\subsection{}
Let $R$ be a discrete valuation ring with perfect residue field $k$. Let $K$ be the fraction field of $R$. Let $\nu \colon K \rightarrow \Z \cup \{ \infty \}$ be the corresponding discrete valuation. Let $C$ be a smooth, projective, geometrically integral curve of genus $g \geq 1$ defined over $K$. Let $S = \Spec R$. Let $X$ be a proper, flat, regular $S$-scheme with generic fiber $C$. The Artin conductor associated to the model $X$ is given by
\[ \Art (X/S) =  \chi(X_{\overline{K}}) - \chi(X_{\overline{k}}) - \delta ,\]
where $\chi$ is the Euler-characteristic for the \'etale topology and $\delta$ is the Swan conductor associated to the $\ell$-adic representation $\Gal (\overline{K}/K) \rightarrow \Aut_{\Q_\ell} (H^1_{\mathrm{et}}(X_{\overline{K}}, \Q_\ell))$ ($\ell \neq \cha k$). The Artin conductor is a measure of degeneracy of the model $X$; it is a non-positive integer that is zero precisely when $X/S$ is smooth or when $g=1$ and $(X_k)_{\mathrm{red}}$ is smooth. If $X/S$ is a regular, semistable model, then $X/S$ equals the number of singular points of the special fiber $X_k$. Let $\Art (C/K)$ denote the Artin conductor associated to the minimal proper regular model of $C$ over $R$.

For hyperelliptic curves, there is another measure of degeneracy defined in terms of minimal Weierstrass equations. Assume that $C$ is hyperelliptic and that $\cha K \neq 2$. An integral Weierstrass equation for $C$ is an equation of the form $y^2+Q(x)y = P(x)$ with $P(x),Q(x) \in R[x]$, such that $C$ is birational to the plane curve given by this equation. The discriminant of such an equation is the non-negative integer $\nu(2^{-4(g+1)} \disc(4P(x)+Q(x)^2))$. A minimal Weierstrass equation is an equation for which the integer $\nu(2^{-4(g+1)} \disc(4P(x)+Q(x)^2))$ is as small as possible amongst all integral equations. The corresponding integer $\nu(\Delta)$ is the minimal discriminant. The minimal discriminant of $C$ is zero precisely when the minimal proper regular model of $C$ is smooth over $S$.

When $g=1$, we have $-\Art (C/K) = \nu(\Delta)$ by the Ogg-Saito formula \cite[p.156, Corollary 2]{saito2}. When $g=2$ and $\cha k \neq 2$, Liu~\cite[p.52, Theoreme 1 and p.53, Theoreme 2]{liup} shows that $-\Art (C/K) \leq \nu(\Delta)$; he also shows that equality can fail to hold. Our main result is an extension of Liu's result to hyperelliptic curves of arbitrary genus under the hypothesis that the Weierstrass points are rational.
\begin{thm}\label{main}
 Let $R$ be a discrete valuation ring with perfect residue field $k$. Assume that $\cha k \neq 2$. Let $K$ be the fraction field of $R$. Let $K^{\mathrm{sh}}$ denote the fraction field of the strict Henselization of $R$. Let $C$ be a hyperelliptic curve over $K$ of genus $g$. Let $\nu \colon K \rightarrow \Z \cup \{\infty\}$ be the discrete valuation on $K$. Assume that the Weierstrass points of $C$ are $K^{\mathrm{sh}}$-rational. Let $S = \Spec R$ and let $\mathcal{X}/S$ be the minimal proper regular model of $C$. Let $\nu(\Delta)$ denote the minimal discriminant of $C$. Then,
 \[ -\Art (\mathcal{X}/S) \leq \nu(\Delta) .\]
\end{thm}

The method of proof is different from the one adopted by Liu in the case of genus $2$ curves. In \cite{saito2}, Saito proves that for a proper regular model $X$ of a smooth curve, $-\Art (X/S)$ equals a certain discriminant defined by Deligne in terms of powers of the relative dualizing sheaf $\omega_{X/S}$. Liu compares the Deligne discriminant of the minimal proper regular model and the minimal discriminant by comparing both of them to a third  discriminant that he defines, that is specific to genus $2$ curves \cite[p.56, Definition 1, p.52, Theoreme 1 and p.53, Theoreme 2]{liup}. In fact, he obtains an exact formula for the difference that can be computed using the Namikawa-Ueno classification of fibers in a pencil of curves of genus $2$ \cite{namiueno}. Since the number of possibilities for the special fiber in a family of curves grows very quickly with the genus (there are already over $120$ types for genus $2$ curves!), we cannot hope to use an explicit classification result and a case by case analysis 
to compare the Deligne discriminant and the minimal discriminant. 

We instead proceed by constructing an explicit proper regular model for the curve $C$ (Section~\ref{construct}). We can immediately reduce to the case where $R$ is a Henselian discrete valuation ring with algebraically closed residue field. 
We may then write a minimal Weierstrass equation for our curve of the form $y^2-f(x)$ where $f$ is a monic polynomial in $R[x]$ that splits completely. If the Weierstrass points of $C$ specialize to distinct points of the special fiber, then the usual compactification of the plane curve $y^2-f(x)$ in weighted projective space over $R$ is already regular. In the general case, we iteratively blow up $\P^1_R$ until the Weierstrass points have distinct specializations. After a few additional blow-ups, we take the normalization of the resulting scheme in the function field of the curve $C$. This gives us a proper regular model for the curve $C$ (Theorem~\ref{Xisregular}) (not necessarily minimal). 

We have the relation $-\Art(X/S) = n(X_s)-1+\tilde{f}$ for a regular model $X$ of the curve $C$, where $n(X_s)$ is the number of components of the special fiber of $X$ and $\tilde{f}$ is an integer that depends only on the curve $C$ and not on the particular regular model chosen. This tells us that to bound $-\Art (X/S)$ for the minimal proper regular model from above, it suffices to bound $-\Art (X/S)$ for some regular model for the curve from above.

In Section~\ref{explicitformula}, we give an explicit formula for the Deligne discriminant for the model we have constructed. After a brief interlude on dual graphs in Section~\ref{dualgraphs}, we restate the formula for the Deligne discriminant using dual graphs (Section~\ref{dualgraphs}). This formula tells us that the Deligne discriminant decomposes as a sum of local terms, indexed by the vertices of the dual graph of the special fiber of the regular model we constructed (Section~\ref{DDanddual}). In Section~\ref{description}, we give a description of the rest of the strategy to prove the main theorem using this formula. The additional ingredients that are necessary are a decomposition of the minimal discriminant into a sum of local terms (Section~\ref{breakupnaive}) and explicit formulae for the local terms in the Deligne discriminant in terms of dual graphs (Section~\ref{localcontribution}). In Section~\ref{comparison}, we show how to compare the Deligne discriminant for the model we have constructed 
and the minimal discriminant locally. To finish the proof, we sum the inequalities coming from all the local terms to obtain $-\Art(X/S) \leq \nu(\Delta)$. As a corollary, we obtain upper bounds on the number of components in the special fiber of the minimal proper regular model (Corollary~\ref{number}). This has applications to Chabauty's method of finding rational points on curves of genus at least $2$ \cite{poonenstoll}.

It might be possible to adapt the same strategy to extend the results to the case of non-rational Weierstrass points. The main difficulties in making this approach work are in understanding the right analogues of the results in Sections~\ref{breakupnaive} and \ref{localcontribution}.

\subsection{Notation}
The invariants $-\Art(X/S)$ and $\nu(\Delta)$ are unchanged when we extend scalars to the strict Henselization. So from the very beginning, we let $R$ be a Henselian discrete valuation ring with algebraically closed residue field $k$. Let $K$ be its fraction field. Assume that $\cha k \neq 2$. Let $\nu \colon K \rightarrow \Z \cup \{\infty\}$ be the discrete valuation on $K$. Let $t$ be a uniformizer of $R$; $\nu(t) = 1$. Let $S = \Spec R$. Let $C$ be a hyperelliptic curve over $K$ with $K$-rational Weierstrass points and genus $g \geq 2$. 

Let $y^2-f(x) = 0$ be an {\textup{\color{blue}integral Weierstrass equation}} for $C$, i.e., $f(x) \in R[x]$ and $C$ is birational to the plane curve given by this equation. The {\textup{\color{blue}discriminant}} of a Weierstrass equation ${\color{blue}d_f}$ equals the discriminant of $f$ considered as a polynomial of degree $2g+2$. A {\textup{\color{blue}minimal Weierstrass equation}} for $C$ is a Weierstrass equation for $C$ such that $\nu(d_f)$ is as small as possible amongst all integral Weierstrass equations for $C$. The {\textup{\color{blue} minimal discriminant $\nu (\Delta)$ of $C$}} equals $\nu(d_f)$ for a minimal Weierstrass equation $y^2-f(x)$ for $C$.

We will first show that we can find a minimal Weierstrass equation such that $f$ is a monic, separable polynomial of degree $2g+2$ in $R[x]$ that splits completely; $f(x) = (x-b_1 )(x-b_2 ) \ldots (x-b_{2g+2} )$ in $R[x]$. Let $y^2-h(x)$ be any minimal Weierstrass equation for $C$. Let $H(x,z) = z^{2g+2}g(x/z)$. Choose a point $\tilde{P} \in \P^1(k)$ that is not a zero of $H$ and let $P \in \P^1(R)$ be a lift of $\tilde{P}$; $P \bmod t = \tilde{P}$. Since $\mathrm{GL}_2(R)$ acts transitively on $\P^1(R)$, we can find $\phi \in \mathrm{GL}_2(R)$ that sends $P$ to $[1:0] \in \P^1_R$. Then, if $F(x,z) = \phi \cdot H(x,z)$, then $F(x,1)$ is of degree $2g+2$ and $u \colonequals F(1,0) \in R$ is a unit. Let $f(x) = u^{-1}F(x,1)$. Since $\cha k \neq 2$ and $R$ is Henselian with algebraically closed residue field, we can find a $u' \in R$ such that $u'^2 = u$. This tells us that by scaling $y$ by $u'$, we obtain a Weierstrass equation $y^2-f(x)$ for $C$ such that $f(x)$ is monic and separable of degree $2g+2$. Since 
$\det \phi$ is a unit in $R$, and the discriminant of $f$ differs from the discriminant of $h$ by a power of $\det \phi$, it follows that $y^2-f(x)$ is a minimal Weierstrass equation for $C$. Fix such an equation. 

For any proper regular curve $Z$ over $S$, we will denote the special fiber of $Z$ by $Z_s$, the generic fiber by $Z_\eta$ and the geometric generic fiber by $Z_{\overline{\eta}}$. We will denote the fraction field of an integral scheme $Z$ by $K(Z)$, the local ring at a point $z$ of a scheme $Z$ by $\O_{Z,z}$ and the unique maximal ideal in $\O_{Z,z}$ by $\mathfrak{m}_{Z,z}$. The reduced scheme attached to a scheme $Z$ will be denoted $Z_{\mathrm{red}}$.

\section{Construction of the regular model}\label{construct}

We first prove a lemma that gives sufficient conditions for the normalization of a regular $2$-dimensional scheme in a degree $2$ extension of its function field to be regular.
\begin{lemma}\label{useful}
Let $Y$ be a regular integral $2$-dimensional scheme and let $f$ be a rational function on $Y$ that is not a square. Assume that the residue field at any closed point of $Y$ is not of characteristic $2$. (Weil divisors make sense on $Y$.) Let $(f) = \sum_{i \in I} m_i \Gamma_i$. Assume that 
 \begin{enumerate}
  \item Any two $\Gamma_i$ for which $m_i$ is odd do not intersect. 
  \item Any $\Gamma_i$ for which $m_i$ is odd is regular. 
 \end{enumerate}

 Then the normalization of $Y$ in $K(Y)(\sqrt{f})$ is regular.
\end{lemma}
\begin{proof}
We will sketch the details of the proof. The construction of the normalization is local on the base. Therefore, it suffices to check that for every closed point $y$ of $Y$, the normalization of the corresponding local ring $\O_{Y,y}$ in $K(Y)(\sqrt{f})$ is regular. There are two cases to consider. 

The first case is when $m_i$ is even for every $\Gamma_i$ that contains $y$. In this case, since $\O_{Y,y}$ is a regular and hence a unique factorization domain, we can write $f = (c_1/c_2)^2u$ for some $c_1,c_2 \in \O_{Y,y} \setminus \{ 0 \}$ and a unit $u \in \O_{Y,y}$. Using the fact that $2$ is a unit in $\O_{Y,y}$ for every $y$, a standard computation then shows that the normalization of $\O_{Y,y}$ in $K(Y)(\sqrt{f})$ is $\O_{Y,y}[z]/(z^2-u)$. From this presentation, we conclude that the normalization is \'{e}tale over $\O_{Y,y}$, and hence regular by {\cite[p.49, Proposition 9]{blr}}. 

The second case is when exactly one of the $m_i$ is odd for the $\Gamma_i$ that contain $y$. Let $a$ be an irreducible element of the unique factorization domain $\O_{Y,y}$, corresponding to the unique $\Gamma_i$ for which $m_i$ is odd. In this case, $f = (c_1/c_2)^2 au$, where $c_1,c_2 \in \O_{Y,y} \setminus \{ 0 \}$ and $u$ is a unit in $\O_{Y,y}$ as before. One can then check that the normalization of $\O_{Y,y}$ in $K(Y)(\sqrt{f})$ is $\O_{Y,y}[z]/(z^2-au)$. Since $\Gamma_i$ is regular at $y$, we can find an element $b \in \O_{Y,y}$ such that $a$ and $b$ generate the maximal ideal of $\O_{Y,y}$. One can then check that $z$ and $b$ generate the unique maximal ideal of $\O_{Y,y}[z]/(z^2-au)$. This implies that $\O_{Y,y}[z]/(z^2-au)$ is regular. 
\end{proof}

In our example, $Y = \P^1_R$ and the rational function $f$ is $(x-b_1 )(x-b_2 ) \ldots (x-b_{2g+2} )$. The divisor of $f$ is just the sum of the irreducible principal horizontal divisors $(x-b_i )$, all appearing with multiplicity $1$ in $(f)$, and the divisor at $\infty$ (the closure of the point at $\infty$ on the generic fiber), with multiplicity $-(2g+2)$. If the $b_i$ belong to distinct residue classes modulo $t$, then the condition in the lemma is satisfied and we get the regular scheme $\mathrm{Proj}\ \frac{R[x,y,z]}{y^2-z^{2g+2}f(x/z)}$. If some of the $b_i$ belong to the same residue class, then the corresponding horizontal divisors would intersect at the closed point on the special fiber given by this residue class and we cannot apply the lemma directly with $Y = \P^1_R$. We will instead apply the lemma to the divisor of $f$ on an iterated blow-up of $\P^1_R$. The generic fiber of this new $Y$ is still $\P^1_K$, so the regular scheme that we obtain will still be a relative $S$-curve with generic 
fiber the hyperelliptic curve we started with. 

We will need another lemma to show that we can resolve the issue discussed above by replacing $\P^1_R$ by an iterated blow-up of $\P^1_R$. The following lemma is a minor modification of \cite[p.64, Lemma 1.4]{liulor}, where we consider irreducible divisors appearing in the divisor of an arbitrary rational function on a model (instead of the rational function $t$) and the order of vanishing of $f$ along these divisors instead. We recover \cite[p.64, Lemma 1.4]{liulor} by taking $f$ to be $t$. 
\begin{lemma}\label{valf}
 Let $Y/R$ be a regular model of a curve $Y_\eta/K$. Let $f$ be a rational function on $Y$. Let $C$ and $D$ be irreducible divisors of $Y$ that appear in the divisor of $f$, and let the order of vanishing of $f$ along $C$ and $D$ be $r_C$ and $r_D$ respectively. Let $y \in Y$ be a closed point, and let $Y'$ denote the model of $Y_\eta$ obtained by blowing up $Y$ at $y$. Let $E \subset Y'$ denote the exceptional divisor.
 \begin{enumerate}
  \item If $y$ is a regular point of $C$ that does not belong to any other irreducible divisor appearing in $(f)$, then the order of vanishing of $f$ along $E$ equals $r_C$.
  \item If $y \in C \cap D$ and does not belong to any other divisors appearing in $(f)$, and if $C$ and $D$ intersect transversally at $y$, then the order of vanishing of $f$ along $E$ is $r_C +r_D$.
 \end{enumerate}
\end{lemma}
\begin{proof}
 Omitted. This can be seen using explicit equations of the blow-up in a neighbourhood of $y$.
\end{proof}

We are now ready to construct the regular model $X$ of $C$. A very similar construction already appears in \cite{kausz} under some additional simplifying hypotheses. The model that is obtained there turns out to be semi-stable. The regular model $X$ that is constructed below is not necessarily semi-stable.

Let $D_i$ be the irreducible principal horizontal divisor $(x-b_i )$ on $\P^1_R$. First blow-up $\P^1_R$ at those closed points on the special fiber where any two of the $D_i$ intersect to obtain a new scheme $\mathrm{Bl}_1(\P^1_R)$. On this scheme, the strict transforms of any two divisors $D_i$ and $D_j$ for which the $b_i$ agree mod $t$ and not mod $t^2$ will no longer intersect. If some of the $b_i$ agree mod $t^2$ as well, then continue to blow-up (that is, now blow up $\mathrm{Bl}_1(\P^1_R)$ at the closed points on the special fiber of $\mathrm{Bl}_1(\P^1_R)$ where any two of the strict transforms of the divisors $(x-b_i)$ intersect, and call the result $\mathrm{Bl}_2(\P^1_R)$). Since the $b_i$ are pairwise distinct, we will eventually end up with a scheme $\mathrm{Bl}_n(\P^1_R)$ where no two of the irreducible horizontal divisors occuring in $(f)$ intersect. We may hope to set $Y$ equal to $\mathrm{Bl}_n(\P^1_R)$, but the divisor of the rational function $f$ might now vanish along some irreducible 
components of the special fiber. 

Lemma~\ref{valf} now tells us that a single blow-up of $\mathrm{Bl}_n(\P^1_R)$ based at a finite set of closed points will ensure that no two components where $f$ vanishes to odd order intersect. Do this as well and call the resulting scheme $Y$. Call an irreducible component of the special fiber of $Y$ {\color{blue} even} if the order of vanishing of $f$ along this component is even. Similarly define {\color{blue} odd} component. Similarly define odd and even components of $\mathrm{Bl}_n(\P^1_R)$.

Recall the notion of a good model as defined in \cite[p.66,1.8]{liulor}. A regular model $Y/\O_K$ of $Y_\eta/K$ is {\color{blue}good} if it satisfies the following two conditions:
\begin{enumerate}
 \item The irreducible components of $Y_s$ are smooth.
 \item Each singular point of $Y_s$ belongs to exactly two irreducible components of $Y_s$ and these components intersect transversally.
\end{enumerate}
The blow-up of a good model at a closed point is again a good model.

The model $Y$ we have constructed is a good model of $\P^1_K$ as it is obtained using a sequence of blow-ups starting from the good model $\P^1_R$ of $\P^1_K$. The model $\mathrm{Bl}_n(\P^1_R)$ is the model we would get using \cite[p.66, Lemma 1.9]{liulor} if we start with the model $\P^1_R$ and the divisor $(f)$ on it. Set $X$ to be equal to the normalization of $Y$ in $K(Y)(\sqrt{f})$. 

\begin{thm}\label{Xisregular}
 The scheme $X/S$ is regular.
\end{thm}
\begin{proof}
 The components of $Y_s$ are smooth and the divisor $(f)$ satisfies the conditions in the statement of Lemma~\ref{useful}. It follows that $X$ is regular.
\end{proof}

We will now prove that $X$ is a good model of $C$ and compute the multiplicities of the components of the special fiber of $X$. Let the divisor of $t$ on $X$ be $\sum m_i \Gamma_i$; here the sum runs over all irreducible components of the special fiber $X_s$ and the $\Gamma_i$ are integral divisors on $X$. Let $\psi$ denote the map $X \rightarrow \mathrm{Bl}_n(\P^1_R)$.

\begin{lemma}\label{construction} \hfill
\begin{enumerate}[\upshape (a)]
\item The scheme $X$ is a good model of $C$.
\item Each $m_i$ is $1$ or $2$. Furthermore, $m_i = 2$ if and only if either
\begin{enumerate}[(i)]
\item $\psi(\Gamma_i)$ is an odd component of $(\mathrm{Bl}_n(\P^1_R))_s$, or,
\item $\psi(\Gamma_i) = \Gamma \cap \Gamma'$ for two distinct odd components $\Gamma$ and $\Gamma'$ of $(\mathrm{Bl}_n(\P^1_R))_s$.
\end{enumerate}
\end{enumerate} 
\end{lemma}
\begin{proof} \hfill
\begin{enumerate}[\upshape (a)]
\item Let $S$ be the set of odd components of $Y_s$ and let $B$ be the divisor $\sum_{\Gamma \in S} \Gamma + \sum_{i=1}^{2g+2} \overline{ \{ b_i \} }$ where $\overline{ \{ b_i \} }$ is the horizontal divisor that is the closure of the point $b_i$ on the generic fiber $\P^1_K$. Since the map $X \rightarrow Y$ is finite of degree $2$, the image of an irreducible component of $X_s$ is an irreducible component of $Y_s$, and there are at most two irreducible components of $X_s$ mapping down to an irreducible component of $Y_s$. All the irreducible components of $Y_s$ are isomorphic to $\P^1_k$. There are two irreducible components of $X_s$ mapping down to a given component of $Y_s$ only when the component of $Y_s$ is an even component that does not intersect any of the irreducible divisors appearing in $B$. In this case the two components in $X_s$ that map down to the given component of $Y_s$ do not intersect, and are isomorphic to $\P^1_k$. In all other cases there is a unique component of $X_s$ mapping down to 
a component of $Y_s$.

Since at most two irreducible components of $Y_s$ pass through any given point of $Y_s$, we see that this implies that at most two irreducible components of $X_s$ pass through any given point of $X_s$. The intersection point $x$ of two irreducible components of $X_s$ has to map to the intersection point $y$ of two irreducible components of $Y_s$. If $y$ is the intersection of two even components, then the map $\psi$ is etale at $x$, so the intersection is still transverse. If $y$ is the intersection of an even and odd component, because the intersection of these components is transverse, we can pick the function $g$ in the proof of Lemma~\ref{useful} to be a uniformizer for the even component. This shows that \'{e}tale locally, the two components that intersect at $x$ are given by the vanishing of $\sqrt{t_ju}$ and $g$ and as these two elements generate the maximal ideal at $x$ \'{e}tale locally, the intersection is transverse once again. For a closed point $x$ on $X_s$ lying on exactly one component $\Gamma$ of $X_s$, the same argument shows that we can choose a system of parameters at the point such that one of them cuts out the component $\Gamma$ of $X_s$. This shows that the irreducible components of $X_s$ are smooth.

\item A repeated application of \cite[p.64, Lemma 1.4]{liulor} tells us that the multiplicity of every irreducible component of $(\mathrm{Bl}_n(\P^1_R))_s$ is $1$. The same lemma tells us that $Y_s$ has a few additional components of multiplicity either $1$ or $2$ - If we blow up the closed point that is the intersection of an odd component of the special fiber of $\mathrm{Bl}_n(\P^1_R)$ with a horizontal divisor appearing in $(f)$, then we get a component of multiplicity $1$ in the special fiber and if we blow up the intersection of two odd components of the special fiber, we get a component of multiplicity $2$. Since $f$ vanishes to an even order along components of multiplicity $2$ in $Y_s$, each $m_i$ is either $1$ or $2$ - It is $1$ if $\Gamma_i$ maps down to an even component of $Y_s$ and its image in $(\mathrm{Bl}_n(\P^1_R))_s$ does not equal the intersection point of two components of the special fiber and it is $2$ otherwise. This is because $\O_{Y,\eta(\psi(\Gamma_i))} \rightarrow \O_{X,\eta(\Gamma_i)}$ is an extension of discrete valuation rings (here $\eta(C)$ for an integral curve $C$ denotes its generic point), and the corresponding extension of fraction fields is of degree $2$. $t$ is a uniformizer in $\O_{Y,\eta(\psi(\Gamma_i))}$, so its valuation above is either $1$ or $2$ depending on whether the extension is ramified at $(t)$ or not. The extension is not ramified if the image of $\Gamma_i$ in $Y$ is an even component. \qedhere
\end{enumerate}
\end{proof}

\section{An explicit formula for the Deligne discriminant}\label{explicitformula}
The Deligne discriminant of the model $X$ is $-\Art (X/S) := -\chi(X_{\overline{\eta}})+\chi(X_s)+\delta$, where $\delta$ is the Swan conductor associated to the $\ell$-adic representation $\mathrm{Gal}\ (\overline{K}/K) \rightarrow \mathrm{Aut}_{\Q_\ell} \ (H^1_{\mathrm{et}}(X_{\overline{\eta}}, \Q_\ell))$ ($\ell \neq \mathrm{char}\ k$) \cite[p.153]{saito2}. 

\begin{lemma}
 \[ -\Art(X/S) = -\chi(X_{\overline{\eta}})+\chi(X_s) =  \sum_i \left( (1-m_i)\chi(\Gamma_i) + \sum_{j \neq i} (m_j-1) \Gamma_i.\Gamma_j \right) + \sum_{i < j} \Gamma_i.\Gamma_j .\] 
\end{lemma}
\begin{proof}
 Since all irreducible components of $X_s$ have multiplicity either $1$ or $2$ in the special fiber and $\mathrm{char}\ k \neq 2$, \cite[p.1044, Theorem 3]{saito1} implies that $\delta = 0$. 

 Using the intersection theory for regular arithmetic surfaces, for a canonical divisor $K$ on $X$, we have
\begin{align*}
 -\chi(X_{\overline{\eta}}) &= 2p_a(X_{\overline{\eta}}) - 2 \\
 &= 2p_a(X_s)-2 \\
 &= X_s.(X_s+K) \\
 &= X_s.K \ \ \ (\textrm{because } X_s \ {\textrm{is a complete fiber}}, X_s.X_s = 0) \\
 &= \sum_i m_i \Gamma_i.K \\
 &= \sum_i m_i (-\chi(\Gamma_i) - \Gamma_i.\Gamma_i) \ \ \ (\textrm{by the adjunction formula applied to the divisor } \Gamma_i)\\
 &= \sum_i \left( -m_i \chi(\Gamma_i) + \sum_{j \neq i} m_j \Gamma_j.\Gamma_i \right) .
\end{align*}
The last equality is obtained from $X_s.\Gamma_i = 0$. 

Let $\lambda \colon \sqcup \Gamma_i \rightarrow (X_s)_{\mathrm{red}}$ be the natural map which is just the inclusion of each $\Gamma_i$ into $(X_s)_{\mathrm{red}}$. Since the $\Gamma_i$ are smooth, \cite[p.151, Theorem 2.6]{lor} tells us that $\chi(X_s) = \chi((X_s)_{\mathrm{red}}) = -\delta_{X_s} + \sum \chi(\Gamma_i)$ where $\delta_{X_s} = \sum_{P \in (X_s)_{\mathrm{red}}} ( |\lambda^{-1}(P)| - 1 )$. In our case $\delta_{X_s}$ is just the number of points where two components of $X_s$ meet. Since the intersections in $X_s$ are all transverse, 
\[ \delta_{X_s} = \sum_{i < j} \Gamma_i.\Gamma_j  = \sum_i \sum_{j \neq i} \Gamma_i.\Gamma_j - \sum_{i < j} \Gamma_i.\Gamma_j .\] 
Putting all this together, we can rewrite $\chi(X_s)$ in the following form
\[ \chi(X_s) = \sum_i \left( \chi(\Gamma_i) - \sum_{j \neq i} \Gamma_i.\Gamma_j \right) + \sum_{i < j} \Gamma_i.\Gamma_j .\]
This expression, together with the formula above for $-\chi(X_{\overline{\eta}})$ gives
\[ -\Art (X/S) = \sum_i \left( (1-m_i)\chi(\Gamma_i) + \sum_{j \neq i} (m_j-1) \Gamma_i.\Gamma_j \right) + \sum_{i < j} \Gamma_i.\Gamma_j . \qedhere \]
\end{proof}

\begin{remark}
 \textup{The formula 
 \[ -\chi(X_{\overline{\eta}}) + \chi(X_s) = \sum_i \left( (1-m_i)\chi(\Gamma_i) + \sum_{j \neq i} (m_j-1) \Gamma_i.\Gamma_j \right) + \sum_{i < j} \Gamma_i.\Gamma_j \]
 holds for any regular $S$ curve $X$ with smooth, projective, geometrically integral generic fiber and whose special fiber is a strict simple normal crossings divisor (i.e., the components themselves might have multliplicities bigger than $1$, but each of the components is smooth, and the reduced special fiber has at worst nodal singularities). We also recover the result that if $X/S$ is regular and semi-stable, then $-\Art (X/S) = \sum_{i < j} \Gamma_i.\Gamma_j$, since in this case $m_i = 1$ for all $i$ and $\delta = 0$ by \cite[p.1044, Theorem~3]{saito1}.}
\end{remark}

\section{Dual graphs}\label{dualgraphs}
By the construction of $X$ we have a sequence of maps $X \rightarrow Y \rightarrow \mathrm{Bl}_n(\P^1_R) \rightarrow \P^1_R$. Let $T_X$ be the dual graph of $X_s$, i.e., the graph with vertices the irreducible components of $X_s$, and an edge between two vertices with an edge if the corresponding irreducible components intersect. Let $T_Y$ be the dual graph of $Y_s$ and $T_B$ the dual graph of $(\mathrm{Bl}_n(\P^1_R))_s$.  For a vertex $v$ of any of the graphs $T_X,T_Y$ or $T_B$, the irreducible component corresponding to the vertex in the respective dual graph will be denoted $\Gamma_v$. Let $\psi_1$ denote the map $X \rightarrow Y$ and let $\psi_2$ the map $Y \rightarrow \mathrm{Bl}_n(\P^1_R)$. Let $\psi = \psi_2 \circ \psi_1$. 

We will denote the vertices of a graph $G$ by $V(G)$. For any $v \in V(G)$, let $N(v)$ (for neighbours of $v$) denote the set of vertices $w$ for which there is an edge between $v$ and $w$. If $G$ is a directed graph and $v \in V(G)$, let $C(v)$ (for children of $v$) denote the set of vertices $w$ for which there is an edge pointing from $v$ to $w$. 

The graph $T_B$ naturally has the structure of a rooted tree (remembering the sequence of blow-ups, i.e., whether the component was obtained as a result of a blow-up at a closed point of the other component). The graph $T_Y$ is obtained from the graph of $T_B$ by attaching some additional vertices between two pre-existing vertices connected by an edge and some additional leaves, so $T_Y$ is also a tree. By virtue of being rooted trees, the edges of $T_B$ and $T_Y$ can  be given a direction (and we choose the direction that points away from the root). 

There is a natural surjective map $\phi_1: V(T_X) \rightarrow V(T_Y)$: if the image of an irreducible component $\Gamma_{v''}$ of $X_s$ under $\psi_1$ is an irreducible component $\Gamma_{v'}$ of $Y_s$ then let $\phi_1(v'') = v'$. If two vertices of $T_X$ are connected by an edge, so are their images in $T_Y$. We can use this surjection to transfer the direction on the edges of $T_Y$ to the edges of $T_X$; this makes $T_X$ a directed graph. Call a vertex of $T_B$ {\color{blue}{odd}} (respectively {\color{blue}{even}}) if the order of vanishing of $f$ along the corresponding component is odd (respectively even). Similarly define odd and even vertices of $T_Y$. This definition is consistent with the earlier definition of odd and even components of $Y$ and $\mathrm{Bl}_n(\P^1_R)$.

\section{Deligne discriminant and dual graphs}\label{DDanddual}
The last term $\sum_{i < j} \Gamma_i.\Gamma_j$ in the Deligne discriminant can be thought of as the sum $\sum_{v'' \in V(T_X)} \left( \sum_{w'' \in C(v'')} \Gamma_{v''}.\Gamma_{w''} \right)$ . We use this observation to decompose the Deligne discriminant as a sum over the vertices of the graph $T_X$. Let $m_{v''}$ be the multiplicity of $\Gamma_{v''}$ in $X_s$. We then have
\[ -\Art (X/S) = \sum_{v'' \in V(T_X)} \left( (1-m_{v''})\ \chi(\Gamma_{v''}) + \sum_{w'' \in N(v'')} (m_{w''} - 1) \Gamma_{v''}.\Gamma_{w''} + \sum_{w'' \in C(v'')} \Gamma_{v''}.\Gamma_{w''} \right) .\]

\section{Description of the strategy}\label{description}
To compare the discriminant $d_f$ of the polynomial $f$ with the valuation of the Deligne discriminant of the model $X$, it would be useful if we could decompose $d_f$ as a sum of local terms. In the next section, we will show that there is a way to decompose the minimal discriminant as a sum over the vertices of $T_B$. There is a simple relation between the irreducible components of $X_s$ and those of $\left( \mathrm{Bl}_n(\P^1_R) \right)_s$ (which we will describe below), so we will be able to compare the two discriminants using this decomposition, by first comparing them locally. 

The image of an irreducible component of $Y_s$ under $\psi_2$ is either an irreducible component of $\left( \mathrm{Bl}_n(\P^1_R) \right)_s$ or a point that lies on exactly one of the irreducible components of $\left( \mathrm{Bl}_n(\P^1_R) \right)_s$ or the intersection point of two irreducible components of $\left( \mathrm{Bl}_n(\P^1_R) \right)_s$.  This induces a surjective map $\phi_2: V(T_Y) \rightarrow V(T_B)$ where the vertex corresponding to an irreducible component of $Y_s$ is mapped either to the vertex corresponding to the unique irreducible component that its image is contained in or to the smaller of the two vertices (by which we mean the vertex closer to the root) corresponding to the two irreducible components that its image is contained in. Let $\phi = \phi_1 \circ \phi_2$.

We have written the Deligne discriminant as $\sum_{v'' \in V(T_X)} \cdots$ and we can rewrite this sum as $\sum_{v \in V(T_B)} (\sum_{v'' \in V(T_X), \phi(v'') = v} \cdots)$, so the Deligne discriminant can be regarded as a sum over the vertices of $T_B$.

The discussion above implies the following lemma, which will be useful later on in an explicit computation of the Deligne discriminant.  
\begin{lemma}\label{adjacency}
Let $v'' \in V(T_X)$. 
\begin{enumerate}[\upshape (a)]
 \item If $w'' \in C(v'')$, then $\phi_1(w'') \in C(\phi_1(v''))$. In particular, if $w'' \in N(v'')$, then $\phi_1(w'') \in N(\phi_1(v''))$.
 \item Let $w'' \in C(v'')$. If $\psi(\Gamma_{w''})$ is a point, then $\phi(w'') = \phi(v'')$ and $\phi(v'')$ is an odd vertex. Otherwise, $\phi(w'') \in C(\phi(v''))$.
\end{enumerate}
\end{lemma}

\section{A decomposition of the minimal discriminant}\label{breakupnaive}
To each vertex $v$ of $T_B$, we want to associate an integer $d(v)$ such that the minimal discriminant equals $\sum_{v \in V(T_B)} d(v)$. We will now define $d(v)$ by inducting on the vertices of $T_B$. 

For the base case, note that if the $b_i $ belong to distinct residue classes modulo $t$, then $\mathrm{Bl}_n(\P^1_R) = \P^1_R$ and $T_B$ is the graph with a single vertex $v$. The minimal discriminant is $0$, so we set $d(v) = 0$.

The scheme $\mathrm{Bl}_n(\P^1_R)$ was obtained as an iterated blow-up of $\P^1_R$ while trying to separate the horizontal divisors $(x-b_i)$ corresponding to the linear factors of $f$. This can be done for any separable polynomial $g \in R[x]$ that splits completely -- let $\mathrm{Bl}(g)$ denote the iterative blow up of $\P^1_R$ that one obtains while trying to separate the divisors corresponding to the linear factors of $g$. With this notation $\mathrm{Bl}(f)$ equals the scheme $\mathrm{Bl}_n(\P^1_R)$ we had above. 

Let $A$ be the set of residues of the $b_i$ modulo $t$. For a residue $a \in A$, let the weight of the residue $a$ ( $\colonequals \wt_a$), be the number of $b_i$ belonging to the residue class of $a$. Observe that the subtrees of the root of $T_B$ are in natural bijection with the residues of weight strictly larger than $1$. 

The minimal discriminant $\nu(\Delta)$ ($=\nu(d_f)$) can be decomposed as follows: 
\begin{align*} 
\nu(d_f) &= \sum_{\substack{a \in A \\ \wt_a > 1}} \nu \left(\prod_{\substack{b_i  \bmod t \ = \ a \\ b_j \bmod t \ = \ a \\ i \neq j}} \left( b_i  - b_j  \right) \right)  \\
&= \sum_{\substack{a \in A \\ \wt_a > 1}} \nu \left(t^{\wt_a(\wt_a-1)}\prod_{\substack{b_i \bmod t \ = \ a \\ b_j \bmod t \ = \ a \\ i \neq j}} \left( \frac{b_i  - b_j }{t} \right) \right)  \\
&= \sum_{a \in A} \wt_a(\wt_a-1) + \sum_{\substack{a \in A \\ \wt_a > 1}} \nu \left(\prod_{\substack{b_i \bmod t \ = \ a \\ b_j \bmod t \ = \ a \\ i \neq j}} \left( \frac{b_i  - b_j }{t} \right) \right) .
\end{align*}

Set $d(\textup{root of }T_B) = \sum_{a \in A} \wt_a(\wt_a-1)$. Pick an element $b_i$ belong to the residue class $a \in A$ of weight strictly bigger than $1$. The subtree corresponding to the residue $a$ can  naturally be identified with the dual graph of $\mathrm{Bl}(g_a)_s$ for the polynomial $\displaystyle{g_a = \prod_{b_j \bmod t \ = \ a} (x- \tfrac{b_j -b_i }{t})}$. Let $d_a$ denote the discriminant of $g_a$. Then,
\[ \nu(d_f) = \sum_{a \in A} \wt_a(\wt_a-1) + \sum_{a \in A} \nu(d_a) .\]
Now recursively decompose $\nu(d_a)$ as a sum over the vertices of the dual graph of $\mathrm{Bl}(g_a)_s$. Identifying the dual graph of $\mathrm{Bl}(g_a)_s$ with the corresponding subtree in $T_B$, this gives us a way to decompose the minimal discriminant as a sum over the vertices of $T_B$.

We will now prescribe a way to attach weights to the vertices of $T_B$ and give an explicit formula for $d(v)$ in terms of these weights.
\subsection{Weight of a vertex}
Suppose $v \in V(T_B)$. Let $T_v$ be the complete subtree of $T_B$ with root $v$. The complete subtree of $T_B$ with root $v$ has as its set of vertices all those vertices of $T_B$ whose path to the root crosses $v$. There is an edge between two vertices in this subtree if there is an edge between them when considered as vertices of $T_B$.  

For each vertex $v$ of $T_B$, define the weight of the vertex $\wt_v$ as follows: Let $J$ be the set of all irreducible components of $(\mathrm{Bl}_n{\P^1_R})_s$ corresponding to the vertices that are in $T_v$. Let $\wt_v$ equal the total number of irreducible horizontal divisors that occur in the divisor $(f)$ in $\mathrm{Bl}_n(\P^1_R)$, not counting the divisor $\overline{\{\infty\}}$, that intersect any of the irreducible components in $J$. 
Thus, if $\Gamma_v$ was obtained as the exceptional divisor in the blow-up of an intermediate iterated blow-up $Z$ between $\mathrm{Bl}_n(\P^1_R)$ and $\P^1_R$ at a smooth closed point of the special fiber $z \in Z_s$, then $\wt_v$ is exactly the number of irreducible horizontal divisors that occur in $(f)$ that intersect $Z_s$ at $z$. This in turn implies the following:
\begin{lemma}\label{weight}
If $v \in V(T_B)$, then $\wt_v \geq 2$.
\end{lemma}

\subsection{Local contribution and weights}
\begin{lemma}\label{formulaford}
 For any vertex $v$ of $T_B$, 
 \[ d(v) = \sum_{w \in C(v)} \wt_w(\wt_w-1) .\] 
\end{lemma}
\begin{proof}
This will once again proceed through an induction on the number of vertices of the tree. For the base case, note that the tree $T_B$ has only one vertex if and only if all the roots of the polynomial $f$ belong to distinct residue classes $\bmod t$ and in this case $d(v) = 0$. Now for the general case. It is clear that the equality holds for the root --- for a residue class $a \in A$ such that $\wt_a > 1$, the weight of the residue class as in the definition is just the weight of the subtree corresponding to the residue class. For any vertex $v$ at depth $1$ (by which we mean one of the nearest neighbours of the root) corresponding to a residue class $a$ such that $\wt_a > 1$, we first observe that the set of roots of the polynomial $\displaystyle{g_a = \prod_{b_j \bmod t \ = \ a} (x- \tfrac{b_j -b_i }{t})}$ corresponding to the residue class $a$ is in natural bijection with a subset of the horizontal divisors of $(f)$ -- namely the ones corresponding to the strict transforms of the divisors $(x-b_j )$ on $\
P^1_R$ for $b_j \bmod t = a$. These are the divisors that intersect the special fiber at one of the irreducible components corresponding to the vertices in this subtree with root $v$. These horizontal divisors are also in bijection with the horizontal divisors of the function $g_a$ different from $\overline{\{\infty\}}$ on $\mathrm{Bl}(g_a)$. The identification of horizontal divisors of $\mathrm{Bl}(g_a)$ and a subset of the horizontal divisors of $\mathrm{Bl}(f)$ is compatible with the identification of the subtree of $T_B$ with the dual graph of $\mathrm{Bl}(g_a)_s$. By this we mean that the set of horizontal divisors intersecting the irreducible component corresponding to any given vertex match up. This tells us that the weight of a vertex of the dual graph of $\mathrm{Bl}(g_a)_s$ equals the weight of the corresponding vertex in $T_B$. Since the lemma holds for the complete subtree at vertex $v$ by induction (where the weights to the vertices of $\mathrm{Bl}(g_a)_s$ are 
assigned using the horizontal divisors of $\mathrm{Bl}(g_a)$), we are done.
\end{proof}

\section{A combinatorial description of the local terms in the Deligne discriminant}\label{localcontribution}
The goal of this section is to obtain explicit formulae (Theorem~\ref{localDel}) for the local terms  appearing in the Deligne discriminant in terms of the combinatorics of the tree $T_B$ (Definition~\ref{lrv}). This involves a careful analysis of the special fiber of $X$ which we present as a series of lemmas.

\begin{lemma}\label{branchlocus} \hfill
\begin{enumerate}[\upshape (a)]
 \item The branch locus of the double cover $\psi_1: X \rightarrow Y$ is the set of all odd components of $Y_s$ along with the strict transforms of the horizontal divisors $(x-b_i)$ on $\P^1_R$. 
 \item If $\Gamma$ is an even component of $Y_s$ and $\Gamma'$ is an irreducible component of the branch locus that intersects $\Gamma$, then $\Gamma.\Gamma' = 1$.
\end{enumerate}
\end{lemma}
\begin{proof} \hfill
\begin{enumerate}[\upshape (a)]
 \item This is clear from the construction of $X$ as outlined in Lemma~\ref{useful}.
 \item From (a), it follows that $\Gamma$ does not belong to the branch locus and $\Gamma'$ is either an odd component of $Y_s$ or the strict transform of the horizontal divisor $(x-b_i)$ on $\P^1_R$ for some $b_i$. 
 
 Suppose $\Gamma'$ is an odd component of $Y_s$. It follows from the construction of $Y$ that if any two irreducible components of $Y_s$ intersect, then they intersect transversally and there is at most one point in the intersection. This implies that $\Gamma.\Gamma' = 1$. 
 
 Suppose $\Gamma'$ is the strict transform of the horizontal divisor $(x-b_i)$ on $\P^1_R$ for some $b_i$. Let $\pi: Y \rightarrow \P^1_R$ be the iterated blow-up map that we obtain from the construction of $Y$. Since $\pi$ is an iterated blow-up morphism,  $\Pic \P^1_R$ is a direct summand of $\Pic Y$, with a canonical projection map $\pi_* : \Pic Y \rightarrow \Pic \P^1_R$. Let $B_i$ denote the Weil divisor $(x-b_i)$ on $\P^1_R$. Then $\pi_* \Gamma' = B_i$.
 \[ 0 < \Gamma.\Gamma' \leq Y_s.\Gamma' = \pi^*(\P^1_R)_s.\Gamma' = (\P^1_R)_s.(\pi_* \Gamma') = (\P^1_R)_s.B_i = 1  . \]
 This implies that $\Gamma.\Gamma' = 1$. \qedhere
\end{enumerate}
\end{proof}

\begin{lemma}\label{oddeven}
 Let $v \in V(T_B)$ and $w \in C(v)$. Then $w(f) = v(f) + \wt_w$. (Here $v(f)$ and $w(f)$ denote the valuation of $f$ in the discrete valuation rings corrresponding to the irreducible divisors $\Gamma_v$ and $\Gamma_{w}$ of $\mathrm{Bl}_n(\P^1_R)$). In particular, if $v$ is even, then $w$ is odd if and only if $\wt_w$ is odd; if $v$ is odd, then $w$ is odd if and only if $\wt_w$ is even. 
\end{lemma}
\begin{proof}
The scheme $\mathrm{Bl}_n(\P^1_R)$ was constructed as an iterated blow-up of $\P^1_R$. There exist intermediate iterated blow-ups $Z'$ and $Z$ of $\P^1_R$ with iterated blow-up maps $\mathrm{Bl}_n(\P^1_R) \rightarrow Z'$, $Z' \rightarrow Z$ and $Z \rightarrow \P^1_R$ such that 
\begin{enumerate}[\upshape (a)]
 \item The scheme $Z'$ is the blow-up of $Z$ at a smooth closed point $z$ of the special fiber $Z_s$. 
 \item The divisor $\Gamma_v \subset \mathrm{Bl}_n(\P^1_R)$ is the strict transform of a vertical divisor $D$ on $Z$ under the morphism $\mathrm{Bl}_n(\P^1_R) \rightarrow Z$. 
 \item $z \in D$.
 \item The divisor $\Gamma_{w} \subset \mathrm{Bl}_n(\P^1_R)$ is the strict transform of $E$ under the morphism $\mathrm{Bl}_n(\P^1_R) \rightarrow Z'$, where $E$ denotes the exceptional divisor of $Z' \rightarrow Z$.
\end{enumerate}
The valuation of $f$ along $E$ equals the multiplicity $\mu_z(f)$ (that is, the largest integer $m$ such that $f \in \mathfrak{m}_{Z,z}^m \setminus \mathfrak{m}_{Z,z}^{m+1}$). There are $\wt_w$ distinct irreducible horizontal divisors of $(f)$ that intersect $Z_s$ at $z$, and $z$ is a smooth point on each of these divisors. This in particular implies that a uniformizer for each of the corresponding discrete valuation rings is in $\mathfrak{m}_{Z,z} \setminus \mathfrak{m}_{Z,z}^2$. From the factorization of $f$ and the fact that $\O_{Z,z}$ is a regular local ring (in particular, a unique factorization domain), one can deduce that $w(f) = \mu_z(f) = v(f) + \wt_w$. This implies that $w(f)$ and $\wt_w$ have the same parity if $v(f)$ is even and have opposite parity if $v(f)$ is odd.
\end{proof}

\begin{defn}\label{lrv}
\textup{Suppose $v \in V(T_B)$. Let $r_v$ be the total number of children of $v$ of odd weight, and let $s_v$ be the total number of children of $v$ of even weight. Let $l'_v$ equal the number of horizontal divisors of $(f)$ different from $\overline{\{ \infty \}}$ passing through $\Gamma_v$ and let $l_v = l'_v+r_v$. For a vertex $v$ of $T_B$ (or of $T_Y$) not equal to the root, let $p_v$ denote the parent of $v$.}  
\end{defn}

Since $\mathrm{Bl}_n(\P^1_R)$ was obtained by iteratively blowing up a regular scheme at smooth rational points on the special fiber, all the components of its special fiber are isomorphic to $\P^1_{k}$ and $X_s$ is reduced. Similarly, all the components of the special fiber of $Y$ are also isomorphic to $\P^1_k$, though $Y_s$ may no longer be reduced.

\begin{lemma}\label{lmod2}
Let $v \in V(T_B)$ be an even vertex. Then $l_v$ is odd if and only if $v$ has an odd parent. In particular, if $v$ is the root, then $l_v$ is even.
\end{lemma}
\begin{proof}
Suppose $v \in V(T_B)$ is even. Then $\psi_2^{-1}(\Gamma_v)$ is a single irreducible component $F$ of $Y_s$ and $\psi_2$ is an isomorphism above a neighbourhood of $\Gamma_v$. Using Lemma~\ref{branchlocus}(b) and the Riemann-Hurwitz formula, we see that the branch locus of $\psi_1$ has to intersect $F$ at an even number of points. Since $v$ is even, Lemma~\ref{branchlocus}(a) and Lemma~\ref{oddeven} imply that $F$ intersects the branch locus at $l_v+1$ points if $v$ has an odd parent, and at $l_v$ points otherwise.
\end{proof}

\begin{lemma}\label{oddY}
 A component of $Y_s$ is odd if and only if it is the strict transform of an odd component of $(\mathrm{Bl}_n(\P^1_R))_s$.
\end{lemma}
\begin{proof}
 The exceptional divisors that arise when we blow up $\mathrm{Bl}_n(\P^1_R)$ to obtain $Y$ are all even by Lemma~\ref{valf}, as every point that is blown up in $\mathrm{Bl}_n(\P^1_R)$ is at the intersection of two odd components.
\end{proof}

\begin{lemma}\label{invimgphi} \hfill
\begin{enumerate}[\upshape (a)]
 \item Let $v' \in V(T_Y)$. Then $\#\phi_1^{-1}(v') = 1$ if $\Gamma_{v'}$ intersects the branch locus of $\psi_1$, and $\# \phi_1^{-1}(v') = 2$ otherwise. If $\# \phi_1^{-1}(v') = 2$, then both irreducible components of $X_s$ corresponding to vertices in $\phi_1^{-1}(v')$ are isomorphic to $\P^1_k$.
 \item  Suppose $v \in V(T_B)$ is an even vertex. Then $\# \phi^{-1}(v)$ is either $1$ or $2$. It is $1$ if and only if $\psi_2^{-1}(\Gamma_v)$ intersects the branch locus of $\psi_1$. If $\# \phi^{-1}(v) = 2$, then both irreducible components of $X_s$ corresponding to vertices in $\phi^{-1}(v)$ are isomorphic to $\P^1_k$.
 \item Suppose $v \in V(T_B)$ is odd. Let $v' \in V(T_Y)$ be the vertex corresponding to the strict transform of $\Gamma_v$ in $Y$. Let
 \begin{align*}
  T_0 &= \{ v' \}, \\
  T_1 &= \{ u' \in \phi_2^{-1}(v) \ | \ \psi_2(\Gamma_{u'}) = \Gamma_v \cap \Gamma_u \ \textup{for some odd } \ u \in C(v) \}, \ \textup{and}, \\
  T_2 &= \bigg\{ {u' \in \phi_2^{-1}(v) \ \left|  \ \begin{aligned} &\psi_2(\Gamma_{u'}) = \Gamma_v \cap H \ \textup{for some irreducible horizontal divisor} \\ &\textup{ $H \neq  \overline{\infty}$ appearing in the divisor of $(f)$}. \end{aligned} \right.  }  \bigg\}  
 \end{align*}
 Let $S_0 = \phi_1^{-1}(T_0) ,S_1 = \phi_1^{-1}(T_1)$ and $S_2 = \phi_1^{-1}(T_2)$. Then
 \begin{enumerate}[(i)]
 \item The sets $T_0,T_1$ and $T_2$ form a partition of $\phi_2^{-1}(v)$. Hence $\{ S_0,S_1,S_2 \}$ is a partition of $\phi^{-1}(v)$.
 \item We have that $\# S_0 = \# T_0 = 1$. Suppose $S_0 = \{ \tilde{v} \}$.  Then $v'$ is odd, $m_{\tilde{v}} = 2$, and
 \[ S_0 =\{ v'' \in \phi^{-1}(v) \ | \ \ \psi(\Gamma_{v''})\ \textup{is not a point} \}. \] 
 \item We have that $\# S_1 = \# T_1 = s_v$. If $v'' \in S_1$, then $m_{v''} = 2$. If $u' \in T_1$, then $u'$ is not a leaf in $T_Y$.
 \item We have that $\# S_2 = \# T_2 = l'_v$. If  $v'' \in S_2$, then $m_{v''} = 1$.
 \item We have that
 \[ T_2 = \{ u' \in \phi^{-1}(v) \ | \ u' \ \textup{is an even leaf of}\ T_Y \} .\]
 \item The map $\phi_1$ induces an isomorphism of graphs between $\phi_2^{-1}(v)$ and $\phi^{-1}(v)$. 
 \item The graph $\phi_2^{-1}(v)$ is a tree with root $v'$ and the graph $\phi^{-1}(v)$ is a tree with root $\psi_1^{-1}(\Gamma_{v'})$.
 \item If $v'' \in \phi^{-1}(v)$, then $\Gamma_{v''} \isom \P^1_k$.
 \end{enumerate}
 \end{enumerate}
\end{lemma}
\begin{proof} \hfill
\begin{enumerate}[\upshape (a)]
 \item All the components of $Y_s$ are isomorphic to $\P^1_k$. Let $v' \in V(T_Y)$. The vertices in $\phi_1^{-1}(v')$ are the irreducible components of $\psi_1^{-1}(\Gamma_{v'})$. If $v'$ is even, then Lemma~\ref{branchlocus}(b) tells us that if $\Gamma_{v'}$ intersects the branch locus at all, it intersects it transversally. Since ramified double covers of $\P^1_k$ are irreducible,  $\psi_1^{-1}(\Gamma_{v'})$ is irreducible if $\Gamma_{v'}$ intersects the branch locus. If $\Gamma_{v'}$ does not intersect the branch locus, as $\P^1_k$ has no connected unramified double covers, we see that $\psi_1^{-1}(\Gamma_{v'})$ has two irreducible components, both of which are isomorphic to $\P^1_k$. This implies that $\# \phi_1^{-1}(v')$ is $1$ if $\Gamma_{v'}$ intersects the branch locus of $\psi_1$ and is $2$ otherwise. 
 
 \item Suppose $v \in V(T_B)$ is even. Then $\psi_2^{-1}(\Gamma_v)$ is a single irreducible component $F$ of $Y_s$ and $\psi_2$ is an isomorphism above a neighbourhood of $\Gamma_v$. Let $v' \in V(T_Y)$ be such that $\Gamma_{v'} = F$. Then $\phi_2^{-1}(v) = \{ v' \}$ and $\phi^{-1}(v) = \phi_1^{-1}(v')$. Apply (a) to $v'$.
 
 \item 
 \begin{enumerate}[(i)]
  \item The component $\Gamma_{v'}$ of $Y_s$ satisfies $\psi_2(\Gamma_{v'}) = \Gamma_{v}$ and it is the only component of $Y_s$ with this property. It follows that $\phi_2(v') = v$. The other components $\Gamma_{u'}$ of $Y_s$ satisfying $\phi_2(u') = v$ are the exceptional divisors of $\psi_2: Y \rightarrow \mathrm{Bl}_n(\P^1_R)$ that get mapped to a point of $\Gamma_v$ that does not also lie on $\Gamma_{p_v}$.
  Since $Y$ is the blow-up of $\mathrm{Bl}_n(\P^1_R)$ at the finite set of points consistsing of the intersection of any two odd components of the special fiber and the intersection of an odd component of the special fiber with an irreducible horizontal divisor $H \neq \overline{\infty}$ appearing in $(f)$, it  follows that $\{ T_0,T_1,T_2 \}$ is a partition of $\phi_2^{-1}(v)$. Since $\phi^{-1}(v) = \phi_1^{-1}(\phi_2^{-1}(v))$, it follows that $\{ S_0,S_1,S_2 \}$ is a partition of $\phi^{-1}(v)$.
  
  \item Lemma~\ref{oddY} tells us $\Gamma_{v'}$ is odd, and Lemma~\ref{branchlocus}(a) tells us that $\psi_1$ is ramified over $\Gamma_{v'}$ and therefore $\psi_1^{-1}(\Gamma_{v'})$ is irreducible, and isomorphic to $\P^1_k$. It follows that $\# S_0 = \# T_0 = 1$. Since $\psi(\Gamma_{\tilde{v}}) = \psi_2(\Gamma_{v'}) = \Gamma_v$ and $v$ is odd, Lemma~\ref{construction}(b) tells us that $m_{\tilde{v}} = 2$. 
  
  Since $\psi(\Gamma_{\tilde{v}}) = \Gamma_v$, it follows that $\psi(\Gamma_{\tilde{v}})$ is not a point. Conversely, suppose $v'' \in \phi^{-1}(v)$ and $\psi(\Gamma_{v''})$ is not a point.  Since $\{ T_0,T_1,T_2 \}$ is a partition of $\phi_2^{-1}(v)$ by (a) and $\psi_2(\Gamma_{u'})$ is a point for $u' \in T_1 \cup T_2$, it follows that $v'' \in \phi_1^{-1}(T_0) = S_0$.
  
  \item For every odd $u \in C(v)$, there exists a unique exceptional curve $E$ of the blow-up $Y \rightarrow \mathrm{Bl}_n(\P^1_R)$ such that if $u' \in V(T_Y)$ is the vertex such that $\Gamma_{u'} = E$, then $u' \in \phi_2^{-1}(v)$ and $\psi_2(\Gamma_{u'}) = \Gamma_v \cap \Gamma_u$. This shows that 
  \[ \# T_1 = \# \textup{odd children of v} = s_v \ (\textup{by Lemma~\ref{oddeven} since $v$ is odd}) . \]
  
  Suppose $u' \in T_1$. Let $w \in C(v)$ be an odd vertex such that $\psi_1(\Gamma_{u'}) = \Gamma_v \cap \Gamma_w$. Let $w' \in V(T_Y)$ be the vertex corresponding to the strict transform of $\Gamma_{w}$ in $Y$.  Then $u' \in C(v')$ and $w' \in C(u')$. In particular, $u'$ is not a leaf. Since $v'$ is odd, Lemma~\ref{branchlocus}(a) and part (a) applied to $u'$ imply that $\# \phi_1^{-1}(u') = 1$. This tells us that $\# S_1 = \# T_1 = s_v$. 
  
  Suppose $v'' \in S_1$. Since $v$ is odd and $\phi_1(v'') \in T_1$, Lemma~\ref{construction}(b) implies that $m_{v''} = 2$.
  
  \item For every irreducible horizontal divisor $H \neq \overline{\infty}$ appearing in the divisor of $(f)$ on $\mathrm{Bl}_n(\P^1_R)$, there exists a unique exceptional curve $E$ of the blow-up $Y \rightarrow \mathrm{Bl}_n(\P^1_R)$ such that if $u' \in V(T_Y)$ is the vertex such that $\Gamma_{u'} = E$, then $u' \in \phi_2^{-1}(v)$ and $\psi_2(\Gamma_{u'}) = \Gamma_v \cap H$. This shows that
  \[ \# T_2 = \# \left\{ \begin{array}{l} \textup{irreducible horizontal divisors $H \neq \overline{\infty}$  appearing in } \\ \textup{$(f)$ on $\mathrm{Bl}_n(\P^1_R)$ that intersect $\Gamma_v$} \end{array} \right\} = l'_v . \]
  
  Suppose $u' \in T_2$. Then $u' \in C(v')$. Since $v'$ is odd, Lemma~\ref{branchlocus}(a) and part (a) applied to $u'$ imply that $\# \phi_1^{-1}(u') = 1$. This tells us that $\# S_2 = \# T_2 = l'_v$. 
  
  Suppose $v'' \in S_2$. Then $\phi_1(v'') \in T_2$. This implies that $\psi(\Gamma_{v''})$ is a point lying on a unique odd component of $(\mathrm{Bl}_n(\P^1_R))_s$, namely $\Gamma_v$. Lemma~\ref{construction}(b) implies that $m_{v''} = 1$.
  
  \item We already observed that $v'$ is the unique vertex of $T_0$ and that it is odd (by Lemma~\ref{oddY}). If $u' \in T_1$, then (iii) implies that $u'$ is not a leaf. This shows 
  \[  \{ u' \in \phi^{-1}(v) \ | \ u' \ \textup{is an even leaf of}\ T_Y \} \subset T_2 .\]
  If $u' \in T_2$, then Lemma~\ref{oddY} implies that $u'$ is even. Since $\Gamma_{u'}$ is the exceptional curve that is obtained by blowing up the point of intersection of an odd component and a horizontal divisor, $u'$ is a leaf. This shows the opposite inclusion. 
  
  \item Parts (ii),(iii),(iv) imply that $\# S_0 = \# T_0, \#S_1 = \# T_1$ and $\# S_2 = \# T_2$. Since $\phi_1$ is a surjection and $\{ T_0,T_1,T_2 \}$ is a partition of $\phi_2^{-1}(v)$, it follows that $\phi_1$ induces a bijection between $\phi^{-1}(v)$ and $\phi_2^{-1}(v)$. 
  
  If $u' \in T_1 \cup T_2$, let $u'' \in \phi^{-1}(v)$ be the unique vertex such that $\phi_1(u'') = u'$. Let $\{ \tilde{v} \} = S_0$. If $u' \in T_1 \cup T_2$, then $u' \in C(v')$. 
  
  If $u' \in T_1 \cup T_2$, then $\Gamma_{\tilde{v}} \cap \Gamma_{u''} = \psi_1^{-1}(\Gamma_{v'} \cap \Gamma_{u'}) \neq \emptyset$. This implies that $u'' \in N(\tilde{v})$ for any $u'' \in S_1 \cup S_2$. If $\tilde{v} \in C(u'')$ for some $u'' \in S_1 \cup S_2$, then Lemma~\ref{adjacency}(a) would imply $v' \in C(u')$. Since $u' \in C(v')$, it follows that $u'' \in C(\tilde{v})$. 
  
  If $u_1',u_2' \in T_1 \cup T_2$, then $\Gamma_{u_1'} \cap \Gamma_{u_2'} = \emptyset$. It now follows from Lemma~\ref{adjacency}(a) and the fact that $\phi_1(u_1''),\phi_1(u_2'') \in T_1 \cup T_2$ that if $u_1'',u_2'' \in S_1 \cup S_2$, then $\Gamma_{u_1''} \cap \Gamma_{u_2''} = \emptyset$. 
  
  Combining the previous three paragraphs, we get that $\phi_1$ induces an isomorphism of graphs between $\phi^{-1}(v)$ and $\phi_2^{-1}(v)$. 
  
  \item The proof of (vi) shows that if $u' \in T_1 \cup T_2$, then $u' \in C(v')$ and that if $u_1',u_2' \in T_1 \cup T_2$, then $\Gamma_{u_1'}$ and $\Gamma_{u_2'}$ do not intersect. It follows that $\phi_2^{-1}(v)$ is a tree with root $v'$. Since (vi) shows $\phi_1$ induces an isomorphism of graphs between $\phi^{-1}(v)$ and $\phi_2^{-1}(v)$, it follows that $\phi^{-1}(v)$ is a tree with root $\psi_1^{-1}(\Gamma_{v'})$.
  
  \item We already observed in the proof of (ii) that if $\{ \tilde{v} \} = S_0$, then $\Gamma_{\tilde{v}} \isom \P^1_k$. 
  
  Suppose $u'' \in S_1$. Let $u' = \phi_1(u'')$. Then $u' \in T_1$. Let $w \in C(v)$ be an odd vertex such that  $\psi_2(\Gamma_{u'}) = \Gamma_v \cap \Gamma_w$. Let $w'$ be the vertex corresponding to the strict transform of $\Gamma_{w}$ in $Y$. Then from the construction of $Y$, it follows that $N(u') = \{ v',w' \}, u' \in C(v')$ and $w' \in C(u')$. Lemma~\ref{oddY} implies that $v'$ and $w'$ are odd and $u'$ is even. Since $\Gamma_{u'} \isom \P^1_k$ and $\Gamma_{u'}$ intersects the branch locus transversally at two points (the points of intersection with $\Gamma_{v'}$ and $\Gamma_{w'}$) by Lemma~\ref{branchlocus}(a,b), the Riemann-Hurwitz formula implies that $\Gamma_{u''} = \psi_1^{-1}(\Gamma_{u'}) \isom \P^1_k$.
  
  Suppose $u'' \in S_2$ and $u' = \phi_1(u'')$. Then $u' \in T_2$. Like in the previous paragraph, we can argue that $\Gamma_{u'}$ intersects the branch locus at exactly two points, corresponding to the point of intersection of $\Gamma_{u'}$ with its odd parent $\Gamma_{v'}$ and the point of intersection of $\Gamma_{u'}$ with an irreducible horizontal divisor $H \neq \overline{\infty}$ appearing in the divisor of $(f)$, and that these intersections are transverse. The Riemann-Hurwitz formula would once again imply $\Gamma_{u''} \isom \P^1_k$. Since (vi) implies that $\{S_0,S_1,S_2\}$ is a partition of $\phi^{-1}(v)$, this completes the proof. \qedhere
 \end{enumerate}  
\end{enumerate} 
\end{proof}

We have the following restatement of Lemma~\ref{construction}(b) using $\phi$ and $\phi_1$.
\begin{lemma}\label{mult2}
 Suppose $v'' \in V(T_X)$. Then $m_{v''} = 2$ if and only if $\phi(v'')$ is odd and $\phi_1(v'')$ is not an even leaf. In particular, if $\phi(v'')$ is even, then $m_{v''} = 1$.
\end{lemma}
\begin{proof}
Lemma~\ref{construction}(b) tells us that $m_{v''} = 2$ if and only if $\psi(\Gamma_{v''})$ is an odd component, or, if $\psi(\Gamma_{v''}) = \Gamma_v \cap \Gamma_w$ for two odd vertices $v,w \in V(T_B)$. 
Let $v = \phi(v'')$. If either of the conditions above hold, it follows from the definition of $\phi$ that the vertex $v$ is odd. So now assume $v$ is odd. Let $\{ S_0,S_1,S_2 \}$ be the partition of $\phi^{-1}(v)$ as in Lemma~\ref{invimgphi}(c). Lemma~\ref{invimgphi}(c)(ii,iii,iv) imply that $m_{v''} = 2$ if and only if $v'' \notin S_2$. Lemma~\ref{invimgphi}(c)(v) then tells us that $v'' \notin S_2$ if and only if $\phi_1(v'')$ is not an even leaf. 

Putting all this together, we get that $m_{v''} = 2$ if and only if $\phi(v'')$ is odd and $\phi_1(v'')$ is not an even leaf.
\end{proof}

\begin{lemma}\label{neigh} 
\begin{enumerate}[\upshape (a)] \hfill
\item Suppose $u'' \in V(T_X)$ and $\psi(\Gamma_{u''})$ is a point. 
 \begin{enumerate}[(i)]
 \item We have that $\# N(u'') = 1$ if $\psi(\Gamma_{u''})$ belongs to a unique odd component of $(\mathrm{Bl}_n(\P^1_R))_s$, and $\# N(u'') = 2$ otherwise. 
 \item If $\# N(u'') = 1$, then $\# C(u'') = 0$. If $\# N(u'') = 2$, then $\# C(u'') = 1$. 
 \item If $w'' \in N(u'')$, then $\phi(w'')$ is an odd vertex.
 \item If $w'' \in N(u'')$, then $m_{w''} = 2$. 
 \end{enumerate}
\item Suppose $u'' \in V(T_X)$, $w'' \in N(u'')$, $\phi(u'')$ is odd and $\phi(w'')$ is even. Then $\psi(\Gamma_{u''})$ is not a point, and the component $\Gamma_{u''}$ is the inverse image under $\psi_1$ of the strict transform of $\Gamma_{\phi(u'')}$.
\end{enumerate}
\end{lemma}
\begin{proof} \hfill
\begin{enumerate}
 \item Let $v = \phi(u'')$. Since $\psi(\Gamma_{u''})$ is a point,  $v$ is odd. Construct the partition $S_0,S_1,S_2$ of $\phi^{-1}(v)$ as in Lemma~\ref{invimgphi}(c). Since $\psi(\Gamma_{u''})$ is a point, Lemma~\ref{invimgphi}(c)(ii) implies that $u'' \in S_1 \cup S_2$. 
 
If $u'' \in S_1$, then $\psi(\Gamma_{u''}) = \Gamma_v \cap \Gamma_w$ for an odd vertex $w \in V(T_B)$. Let $v',w'$ be the vertices in $T_Y$ corresponding to the strict transforms of $\Gamma_v$ and $\Gamma_w$ respectively. Since $v$ and $w$ are odd, Lemma~\ref{oddY} tells us that $v'$ and $w'$ are odd. Then $N(\phi_1(u'')) = \{ v',w' \}$. By Lemma~\ref{invimgphi}(a), the vertices $v',\phi_1(u''),w'$ of $T_Y$ each have exactly one preimage under under $\phi_1$. Let $v'',w'' \in V(T_X)$ such that $\phi_1(v'') = v'$ and $\phi_1(w'') = w'$. The unique point $\Gamma_{v'} \cap \Gamma_{\phi_1(u'')}$ has exactly one preimage under $\psi_1$ and therefore lies on both $\Gamma_{v''}$ and $\Gamma_{u''}$. Similarly, $\Gamma_{u''} \cap \Gamma_{w''}$ is nonempty. Lemma~\ref{adjacency}(a) now tells us that $N(u'') = \{v'',w'' \}$. This implies that $\# N(u'') = 2$ and $\# C(u'') = 1$. We also have $\phi(v'') = v$ and $\phi(w'') = w$, and both $v$ and $w$ are odd vertices. Since $\phi(v'')$ is odd and $\phi_1(v'') = v'$ is 
odd, Lemma~\ref{mult2} tells us that $m_{v''} = 2$. Similarly, we can show $m_{w''} = 2$.  
 
If $u'' \in S_2$, then Lemma~\ref{invimgphi}(c)(v) implies that $u' := \phi_1(u'')$ is an even leaf of $T_Y$. Lemma~\ref{invimgphi}(c)(vii) shows $u'$ has a parent. Let $v' = p_{\phi_1(u'')}$ and $v = \phi_2(v')$. Lemma~\ref{invimgphi}(c)(ii,vii) imply that $v'$ is an odd vertex corresponding to the strict transform of $\Gamma_v$ in $Y$, and $\# \phi_1^{-1}(v') = 1$. Let $v'' \in V(T_X)$ be such that $\phi_1(v'') = v'$. Then the unique point in $\Gamma_{v'} \cap \Gamma_{u'}$ has exactly one preimage under $\psi_1$ and this preimage is contained in $\Gamma_{v''} \cap \Gamma_{u''}$. Lemma~\ref{adjacency} now tells us that $\# N(u'') = 1$ and $\# C(u'') = 0$. Lemma~\ref{oddY} implies that $\phi(v'') = \phi_2(v') = v$ is odd. Since $\phi(v'') = v$ is odd and $\phi_1(v'') = v'$ is also odd, Lemma~\ref{mult2} implies that $m_{v''} = 2$.
 
The definitions of $T_1,T_2,S_1,S_2$ in Lemma~\ref{invimgphi}(c) show that the vertices in $S_1$ are exactly the ones corresponding to irreducible components of $X_s$ whose images under $\psi$ are contained in two odd components of $(\mathrm{Bl}_n(\P^1_R))_s$ and the vertices in $S_2$ are the ones corresponding to irreducible components of $X_s$ whose images under $\psi$ are contained in exactly one odd component. 

\item Suppose $u'' \in V(T_X)$, $w'' \in N(u'')$, $\phi(u'')$ is odd and $\phi(w'')$ is even. Then part (a) of this lemma tells us that $\psi(\Gamma_{u''})$ is not a point. If $S_0,S_1,S_2$ is the partition of $\phi^{-1}(\phi(u''))$ as in Lemma~\ref{invimgphi}(c), then Lemma~\ref{invimgphi}(c)(ii) implies that $u'' \in S_0$ since $\psi(\Gamma_{u''})$ is not a point. As $S_0$ has a unique vertex, and this vertex corresponds to the inverse image under $\psi_1$ of the strict transform of $\Gamma_{\phi(u'')}$, we are done.
\qedhere
\end{enumerate}
\end{proof}

\begin{lemma}\label{intersect}
 Let $v'',w'' \in V(T_X)$. Then $\Gamma_{v''}.\Gamma_{w''} \in \{ 0,1,2 \}$. Let $v = \phi(v''), w = \phi(w''), v' = \phi_1(v'')$ and $w' = \phi_1(w'')$. Then $\Gamma_{v''}.\Gamma_{w''} = 2$ if and only if
 \begin{enumerate}[(i)]
  \item both $v$ and $w$ are even,
  \item the vertices $v$ and $w$ are neighbours of each other, and,
  \item both $\Gamma_{v'}$ and $\Gamma_{w'}$ intersect the branch locus of $\psi_1$.
\end{enumerate}
\end{lemma}
\begin{proof}
Lemma~\ref{construction}(b) tells us that all intersections in $X_s$ are transverse, so the the number of points in the intersection of any two irreducible components in $X_s$ equals their intersection number.

Let $v'',w'' \in V(T_X)$. Then $\Gamma_{v''} \cap \Gamma_{w''} \subset \psi_1^{-1}(\Gamma_{v'} \cap \Gamma_{w'})$. Since $\psi_1$ is finite of degree $2$, any point of $Y$ has at most two preimages under $\psi_1$ and therefore $\# \psi_1^{-1}(\Gamma_{v'} \cap \Gamma_{w'}) \leq 2 \# \Gamma_{v'} \cap \Gamma_{w'}$. The set $\Gamma_{v'} \cap \Gamma_{w'}$ has at most one point since the dual graph $T_Y$ of $Y_s$ is a tree. This implies that $\# \Gamma_{v'} \cap \Gamma_{w'} \leq 1$. Putting these together, we get
\begin{align*}
 \Gamma_{v''}.\Gamma_{w''} &= \# \Gamma_{v''} \cap \Gamma_{w''} \\
 &\leq \#\psi_1^{-1}(\Gamma_{v'} \cap \Gamma_{w'}) \\
 &\leq 2 \ \# \Gamma_{v'} \cap \Gamma_{w'} \\
 &\leq 2.1 \\
 &= 2 .
\end{align*} 
It follows that $\Gamma_{v''}.\Gamma_{w''} \in \{0,1,2\}$.

Suppose that the three conditions in the lemma hold. Then, conditions $(i)$ and $(ii)$ imply that $\Gamma_v \cap \Gamma_w$ is nonempty and consists of a single point, say $b$. Then the strict transforms of $\Gamma_v$ and $\Gamma_w$ are $\Gamma_{v'}$ and $\Gamma_{w'}$ respectively and the map $\psi_2$ is an isomorphism above a neighbourhood of $\Gamma_{v} \cup \Gamma_{w}$. Let $y$ be the unique point in $\Gamma_{v'} \cap \Gamma_{w'}$. As $T_Y$ is a tree, the point $y$ does not lie on any other component of $Y_s$ except $\Gamma_{v'}$ and $\Gamma_{w'}$. Lemma~\ref{oddY} tells us that $v'$ and $w'$ are even. Lemma~\ref{branchlocus}(a) now tells us that the point $y$ has two preimages under $\psi_1$. Since $\Gamma_{v'}$ and $\Gamma_{w'}$ intersect the branch locus, their inverse images under $\psi_1$ are irreducible. This tells us $\Gamma_{v''} = \psi_1^{-1}(\Gamma_{v'})$ and $\Gamma_{w''} = \psi_1^{-1}(\Gamma_{w'})$. Then $\psi_1^{-1}(\Gamma_{v'} \cap \Gamma_{w'}) = \Gamma_{v''} \cap \Gamma_{w''}$. 
\begin{align*}
 \Gamma_{v''}.\Gamma_{w''} &= \# \Gamma_{v''} \cap \Gamma_{w''} \\
 &= \# \psi_1^{-1}(\Gamma_{v'} \cap \Gamma_{w'}) \\
 &= \# \psi_1^{-1}(y) \\
 &= 2.
\end{align*}
 
Now assume $\Gamma_{v''}.\Gamma_{w''} = 2$. Since the intersections in $X_s$ are transverse, the set $\Gamma_{v''} \cap \Gamma_{w''}$ has two points, say $x_1$ and $x_2$. Then, $\psi_1(x_1)$ and $\psi_1(x_2)$ must lie in $\Gamma_{v'} \cap \Gamma_{w'}$. Since any two components of $Y_s$ cannot intersect at more than one point, this tells us that $\psi_1(x_1) = \psi_1(x_2)$. Call this point of intersection $y$. Since $y$ has two preimages under $\psi_1$, it cannot lie on the branch locus of $\psi_1$. Lemma~\ref{branchlocus}(a) tells us that $v'$ and $w'$ must both be even. Since $\psi(\Gamma_{v''}) = \Gamma_v$, it follows that $\psi(\Gamma_{v''})$ is not a point. Similarly $\psi(\Gamma_{w''}) = \Gamma_w$ is not a point. Either $w'' \in C(v'')$ or $v'' \in C(w'')$, and Lemma~\ref{adjacency}(b) tells us that in both cases $v$ and $w$ are neighbours of each other. If $\Gamma_{v'}$ did not intersect the branch locus, then Lemma~\ref{invimgphi}(a) implies that $\psi_1^{-1}(\Gamma_{v'})$ must have two disjoint 
irreducible components, one of which is the $\Gamma_{v''}$ we started with. Let $\tilde{v}'' \in V(T_X)$ be the other. Then there is exactly one point of $\psi_1^{-1}(y)$ in each $\Gamma_{v''}$ and $\Gamma_{\tilde{v}''}$. This contradicts the fact that $\Gamma_{v''}$ has both points of $\psi_1^{-1}(y)$. A similar argument shows that $\Gamma_{w'}$ intersects the branch locus.
\end{proof}

We now make some definitions motivated by Sections $6$ and $7$. For $v'' \in V(T_X)$, define
\[ \delta(v'') =  (1-m_{v''})\ \chi(\Gamma_{v''}) + \sum_{w'' \in N({v''})} (m_{w''} - 1) \Gamma_{v''}.\Gamma_{w''} +  \sum_{w'' \in C({v''})} \Gamma_{v''}.\Gamma_{w''}   .\]
Let $v \in V(T_B)$. Define 
\[ D(v) = \sum_{v'' \in \phi^{-1}(v)} \delta(v''). \]

\subsection{Computation of $D(v)$ for an even vertex $v$}
Suppose $v \in V(T_B)$ is an even vertex. We define $D_0(v), D_1(v), D_2(v)$ as follows.
\begin{align*}
D_0(v) &= \sum_{v'' \in \phi^{-1}(v)} (1-m_{v''})\ \chi(\Gamma_{v''}) .\\
D_1(v) &= \sum_{v'' \in \phi^{-1}(v)} \ \ \sum_{w'' \in N({v''})} (m_{w''} - 1) \Gamma_{v''}.\Gamma_{w''} .\\
D_2(v) &= \sum_{v'' \in \phi^{-1}(v)} \ \ \sum_{w'' \in C({v''})} \Gamma_{v''}.\Gamma_{w''} 
\end{align*}
Then, $D(v) = D_0(v)+D_1(v)+D_2(v)$. We will now compute $D_i(v)$ for each $i \in \{0,1,2\}$ in terms of $l_v,r_v$ and $s_v$.

\begin{lemma}\label{evenfirst}
 Suppose $v \in V(T_B)$ is even. Then, $ D_0(v) = 0 $.
\end{lemma}
\begin{proof}
Suppose $v$ is an even vertex. Lemma~\ref{mult2} implies that $m_{v''} = 1$ for every $v'' \in \phi^{-1}(v)$ and therefore,
\[ D_0(v) = \sum_{v'' \in \phi^{-1}(v)} (1-m_{v''})\ \chi(\Gamma_{v''}) = 0 . \qedhere \]
\end{proof}

\begin{lemma}\label{evenprelim}
Suppose $v \in V(T_B)$ is even. Let $v'' \in \phi^{-1}(v)$ and $w'' \in N(v'')$. Let $v' = \phi_1(v''), w' = \phi_1(w'')$ and $w = \phi(w'')$. 
\begin{enumerate}[\upshape (a)]
 \item The vertex $v'$ is even and $\phi_2^{-1}(v) = \{ v' \}$.
 \item The multiplicity $m_{w''} = 2$ if and only if $w$ is odd.
 \item If $v'' \in C(w'')$, then $v \in C(w)$. If $w'' \in C(v'')$, then $w \in C(v)$. In particular, $w \in N(v)$.
 \item If $r_v = 0$ and $l_v$ is even, then every neighbour of $v$ is even.
 \item The branch locus of $\psi_1$ intersects $\Gamma_{v'}$ at $l_v+(l_v \bmod 2)$ points, and all these intersections are transverse.
 \item If $l_v = 0$, then $\Gamma_{v'}$ does not intersect the branch locus of $\psi_1$ and $\# \phi^{-1}(v) = 2$. 
 \item If $l_v \neq 0$, then $\Gamma_{v'}$ intersects the branch locus of $\psi_1$, $\# \phi^{-1}(v) = 1$ and $\phi^{-1}(v) = \{ v'' \}$.
 \item If $w$ is odd, then $\Gamma_{v''}.\Gamma_{w''} = 1$.
 \item Suppose $u \in N(v)$ is odd. Then there exists a unique $u'' \in \phi^{-1}(u)$ such that $u'' \in N(v'')$. If $u \in C(v)$, then $u'' \in C(v'')$. If $v \in C(u)$, then $v'' \in C(u'')$.
 \item Suppose $l_v \neq 0$, $w'' \in C(v'')$ and $w$ is even. Then,  $\# \phi^{-1}(w) \in \{1,2\}$. If $\# \phi^{-1}(w) = 1$, then  $\Gamma_{v''}.\Gamma_{w''} = 2$. If $\# \phi^{-1}(w) = 2$, then $\Gamma_{v''}.\Gamma_{w''} = 1$.
 \item Suppose $l_v \neq 0$ and $u \in C(v)$ is even. If $u'' \in \phi^{-1}(u)$, then $u'' \in C(v'')$. 
 \item If $l_v = 0$, then $\Gamma_{v''}.\Gamma_{w''} = 1$.
 \item Suppose $l_v = 0$ and $u \in C(v)$ is even. If $\phi^{-1}(u) = \{ u'' \}$, then $u'' \in C(v'')$. If $\phi^{-1}(u) = \{ u_1'',u_2'' \}$, then, after possibly interchanging $u_1''$ and $u_2''$, we have that
 $u_1'' \in C(v'')$ and $\Gamma_{v''}.\Gamma_{u_2''} = 0$. 
\end{enumerate}
\end{lemma}
\begin{proof} \hfill
\begin{enumerate}[\upshape (a)] 
  \item Since $\phi_2(v') = \phi(v'') = v$ and $v$ is even, Lemma~\ref{oddY} tells us that $v'$ is even. 
 
  \item First assume $w$ is odd. Since $v'$ is even, Lemma~\ref{neigh}(b) implies that $\Gamma_{w''}$ is the preimage under $\psi_1$ of the strict transform of $\Gamma_w$ in $Y$. In particular, Lemma~\ref{oddY} tells us that $w'$ is odd, and therefore not an even leaf. Lemma~\ref{mult2} applied to $w''$ then implies that $m_{w''} = 2$. 
  
  Conversely, assume $m_{w''} = 2$.  Lemma~\ref{mult2} applied to $w''$ implies that $w$ is odd. 
  
  \item If $w$ is odd, since $v$ is even, Lemma~\ref{neigh}(b) tells us that $\psi(\Gamma_{w''})$ is not a point. If $w$ is even, then $\psi(\Gamma_{w''})$ is not a point. Since $v$ is even,  $\psi(\Gamma_{v''})$ is not a point. Since $w'' \in N(v'')$, either $v'' \in C(w'')$ or $w'' \in C(v'')$. Since both $\psi(\Gamma_{v''})$ and $\psi(\Gamma_{w''})$ are not points, Lemma~\ref{adjacency}(b) tells us that in the first case $v \in C(w)$ and in the second case $w \in C(v)$. Both of these imply $w \in N(v)$. 
  
  \item Suppose $r_v = 0$ and $l_v$ is even. Since $v$ is even and $r_v = 0$, Lemma~\ref{oddeven} implies that every child of $v$ is even. Since $l_v$ is even, Lemma~\ref{lmod2} implies that $v$ does not have an odd parent. Therefore every neighbour of $v$ is even.
  
  \item Lemma~\ref{branchlocus}(a) and Lemma~\ref{oddY} tell us that $\Gamma_{v'}$ does not belong to the branch locus since $\phi_2(v') = v$, which is even. Lemma~\ref{branchlocus}(b) tells us that any component of the branch locus that intersects $\Gamma_{v'}$, intersects it transversally. 
  \begin{itemize}
   \item Lemma~\ref{branchlocus}(a) tells us that the components of the branch locus are the odd components of $Y_s$ and the irreducible horizontal divisors appearing in $(f)$ different from $\overline{\infty}$.
   \item Lemma~\ref{oddY} tells us that the odd components of $Y_s$ are the strict transforms of odd components of $(\mathrm{Bl}_n(\P^1_R))_s$. 
   \item Since $v$ is even, the map $\psi_2$ induces an isomorphism above a neighbourhood of $\Gamma_v$. \end{itemize}
   Therefore, the number of components of the branch locus intersecting $\Gamma_{v'}$ is the number of odd neighbours of $v$ added to the number of horizontal divisors different from $\overline{\infty}$ appearing in the divisor of $(f)$ that intersect $\Gamma_v$. The latter number is $l'_v$. Since $v$ is even, Lemma~\ref{oddeven} tells us that the number of odd children of $v$ is $r_v$. Lemma~\ref{lmod2} tells us that the number of odd parents of $v$ is $(l_v \bmod 2)$. Since $l'_v+r_v+(l_v \bmod 2) = l_v + (l_v \bmod 2)$, the branch locus intersects $\Gamma_{v'}$ at $l_v + (l_v \bmod 2)$ points.
  
  \item Suppose $l_v = 0$. Then $l_v + (l_v \bmod 2) = 0$. Part (e) tell us that $\Gamma_{v'}$ does not intersect the branch locus of $\psi_1$. Since $v$ is even, Lemma~\ref{invimgphi}(b) implies that $\# \phi^{-1}(v) = 2$.
%
  
  \item Suppose $l_v \neq 0$. Then $l_v + (l_v \bmod 2) \neq 0$. Part (e) tells us that $\Gamma_{v'}$ intersects the branch locus of $\psi_1$. Lemma~\ref{invimgphi}(b) then implies that $\# \phi^{-1}(v) = 1$. It follows that $\phi^{-1}(v) = \{ v'' \}$.
%
%
  
  \item Suppose $w$ is odd. Since $w$ is odd, Lemma~\ref{intersect} tells us that $\Gamma_{v''}.\Gamma_{w''} < 2$. On the other hand, since $w'' \in N(v'')$, it follows that $\Gamma_{v''}.\Gamma_{w''} \geq 1$. 
  
  \item Suppose $u \in N(v)$ is odd. Let $u'$ be the vertex corresponding to the strict transform of $\Gamma_u$ in $Y$. As $u \in N(v)$ and $\psi_2$ is an isomorphism above a neighbourhood of $\Gamma_{v}$, it follows that $u' \in N(v')$. In fact, this shows that if $u \in C(v)$, then $u' \in C(v')$; if $v \in C(u)$, then $v' \in C(u')$.
  
  Lemma~\ref{oddY} shows that $u'$ is odd. Lemma~\ref{branchlocus}(a) and Lemma~\ref{invimgphi}(a) applied to $u'$ show that there is a unique $u''$ in $V(T_X)$ such that $\phi_1(u'') = u'$. Part (a) tells us that $v'$ is even and $\phi_2^{-1}(v) = \{ v' \}$. Since $\Gamma_{v'}$ intersects $\Gamma_{u'}$ and $u'$ is odd, Lemma~\ref{branchlocus}(a) and Lemma~\ref{invimgphi}(a) applied to the even vertex $v'$ tell us that $\phi^{-1}(v) = \phi_1^{-1}(v') = \{ v'' \}$. Since $\psi_1^{-1}(\Gamma_{u'}) = \Gamma_{u''}$ and $\psi_1^{-1}(\Gamma_{v'}) = \Gamma_{v''}$, it follows that $\Gamma_{u''} \cap \Gamma_{v''} = \psi_1^{-1}(\Gamma_{u'} \cap \Gamma_{v'})$. Since $\psi_1$ is surjective and $\Gamma_{u'} \cap \Gamma_{v'}$ is nonempty, it follows that $u'' \in N(v'')$. We also have $\phi(u'') = \phi_2(u') = u$. This proves the existence of $u'' \in \phi^{-1}(u)$ such that $u'' \in N(v'')$.
  
  Suppose that we are given $u'' \in \phi^{-1}(u)$ such that $u'' \in N(v'')$. Since $v$ is even and $u$ is odd, Lemma~\ref{neigh}(b)  forces $u''$ to be the inverse image under $\psi_1$ of the strict transform of $\Gamma_u$ in $Y$. This proves uniqueness.
  
  Lemma~\ref{adjacency}(a) tells us that if $v'' \in C(u'')$, then $v' \in C(u')$. If $u \in C(v)$, then $u' \in C(v')$ and therefore $u'' \in C(v'')$. Similarly, one can show that if $v \in C(u)$, then $v'' \in C(u'')$.
  
  \item Part (g) tells us that $\Gamma_{v'}$ intersects the branch locus of $\psi_1$. Since $w$ is even, Lemma~\ref{invimgphi}(b) implies that $\# \phi^{-1}(w) \in \{1,2\}$. Since $v$ and $w$ are even,  $\# \psi^{-1}(\Gamma_v \cap \Gamma_{w}) = 2$. Since $w'' \in C(v'')$, Lemma~\ref{intersect} tells us that $1 \leq \Gamma_{v''}.\Gamma_{w''} \leq 2$. We have that $v$ and $w$ are even, $w \in C(v)$ (by (c)) and that $\Gamma_{v'}$ intersects the branch locus; thus, Lemma~\ref{intersect} implies that $\Gamma_{v''}.\Gamma_{w''} = 2$ if $\Gamma_{w'}$ intersects the branch locus, and $\Gamma_{v''}.\Gamma_{w''} = 1$ if it does not. Lemma~\ref{invimgphi}(b) applied to $w$ tells us that this can be restated as follows: If $\# \phi^{-1}(w) = 1$, then  $\Gamma_{v''}.\Gamma_{w''} = 2$; if $\# \phi^{-1}(w) = 2$, then $\Gamma_{v''}.\Gamma_{w''} = 1$.
  
   \item Let $u' \in V(T_Y)$ be the vertex corresponding to the strict transform of $\Gamma_u$ in $Y$. Let $u'' \in \phi^{-1}(u)$.
   \begin{itemize}
    \item Part (g) tells us that $\Gamma_{v'}$ intersects the branch locus of $\psi_1$ and $\phi^{-1}(v) = \{ v'' \}$. Therefore $\psi_1^{-1}(\Gamma_{v'}) = \Gamma_{v''}$.
    \item Since $\psi_2$ is an isomorphism above a neighbourhood of $\Gamma_v$, we have that $u' \in C(v')$. In particular, $\Gamma_{u'} \cap \Gamma_{v'} \neq \emptyset$.
    \item The map $\psi_1$ restricts to a surjection $\Gamma_{u''} \rightarrow \Gamma_{u'}$.
   \end{itemize}
    These three facts together imply that $\Gamma_{u''} \cap \Gamma_{v''}$ is not empty. In particular,  $u'' \in N(v'')$. If $v'' \in C(u'')$, then Lemma~\ref{adjacency}(a) would imply $v' \in C(u')$.  Since $u' \in C(v')$, Lemma~\ref{adjacency}(a) implies that $u'' \in C(v'')$.  
  
  \item Suppose $l_v = 0$. Part (f) tells us that $\Gamma_{v'}$ does not intersect the branch locus. Lemma~\ref{intersect} applied to the pair $v'',w''$ tells us $\Gamma_{v''}.\Gamma_{w''} < 2$. On the other hand, since $w'' \in N(v'')$, we have that $\Gamma_{v''}.\Gamma_{w''} \geq 1$. Therefore,  $\Gamma_{v''}.\Gamma_{w''} = 1$.
  
  \item Let $u' \in V(T_Y)$ be the vertex corresponding to the strict transform of $\Gamma_u$ in $Y$. Since $\psi_2$ is an isomorphism above a neighbourhood of $\Gamma_v$, we get that $u' \in C(v')$.  
  
  Suppose $\phi^{-1}(u) = \{ u'' \}$. Since $\psi_1^{-1}(\Gamma_{u'}) = \Gamma_{u''}$ and $\psi_1$ restricts to a surjection $\Gamma_{v''} \rightarrow \Gamma_{v'}$, an appropriate modification of the argument in part(j) tells us that $u'' \in C(v'')$.
  
  Suppose $\phi^{-1}(u) = \{ u_1'',u_2'' \}$. Then, Lemma~\ref{invimgphi}(a) implies that $\Gamma_{u'}$ does not intersect the branch locus. Part (f) implies that $\Gamma_{v'}$ does not intersect the branch locus. This implies that the map $\psi_1$ is \'{e}tale above a neighbourhood of $\Gamma_{v'} \cup \Gamma_{u'}$. Since $\P^1_k$ has no connected \'{e}tale covers, this implies that $\psi_1^{-1}(\Gamma_{v'} \cup \Gamma_{u'})$ has two connected components, each of which maps isomorphically on to $\Gamma_{v'} \cup \Gamma_{u'}$ via $\psi_1$. This finishes the proof.  \qedhere
 \end{enumerate}
\end{proof}

\begin{lemma}\label{evensecond}
 Suppose $v \in V(T_B)$ is even. Then, $ D_1(v) = (l_v \bmod 2)+r_v$. (Here and subsequently $l_v \bmod 2$ is an integer in $\{0,1\}$. It is $0$ if $l_v$ is even and $1$ if $l_v$ is odd.)
\end{lemma}
\begin{proof}
 Suppose $v \in V(T_B)$ is even. 
 We break up the computation of $D_1(v)$ into two cases:
\begin{flushleft} $\textbf{\textup{Case 1}} \colon \mathbf{l_v = 0}$ \end{flushleft}
 In this case, 
 \begin{align*} 
D_1(v) &= \sum_{v'' \in \phi^{-1}(v)}  \sum_{w'' \in N(v'')} (m_{w''} - 1) \Gamma_{v''}.\Gamma_{w''}  \\
&= \sum_{v'' \in \phi^{-1}(v)} {\sum_{\substack{w'' \in N(v'') \\ \phi(w'') \ \mathrm{even}}}} (m_{w''} - 1) \Gamma_{v''}.\Gamma_{w''}  \ \ \ \ (\mathrm{by \ Lemma~\ref{evenprelim}(d) \ since \ r_v =l_v = 0 }) \\
&= \sum_{v'' \in \phi^{-1}(v)} {\sum_{\substack{w'' \in N(v'') \\ \phi(w'') \ \mathrm{even}}}} (1 - 1) \Gamma_{v''}.\Gamma_{w''} \ \ \ \ (\mathrm{by \ Lemma~\ref{evenprelim}(b)}) \\
&=  0 \\
&= (l_v \bmod 2)+r_v\ \ \ \ (\textup{since $l'_v$ and $r_v$ are nonnegative, $r_v = 0$}) . 
\end{align*}

 \begin{flushleft} $\textbf{\textup{Case 2}} \colon \mathbf{l_v \neq 0}$ \end{flushleft}
 In this case, Lemma~\ref{evenprelim}(g) implies that $\# \phi^{-1}(v) = 1$. Let $\phi^{-1}(v) = \{ v'' \}$. Then,
\begin{align*}
 D_1(v) &= \sum_{\tilde{v}'' \in \phi^{-1}(v)} \sum_{w'' \in N(\tilde{v}'')} (m_{w''} - 1) \Gamma_{\tilde{v}''}.\Gamma_{w''} \\
 &= \sum_{w'' \in N(v'')} (m_{w''} - 1) \Gamma_{v''}.\Gamma_{w''} \\
 &= {\sum_{\substack{w'' \in N(v'') \\ \ \phi(w'') \ \mathrm{odd}}} (m_{w''} - 1) \Gamma_{v''}.\Gamma_{w''}} + {\sum_{\substack{w'' \in N(v'') \\ \phi(w'') \ \mathrm{even}}} (m_{w''} - 1) \Gamma_{v''}.\Gamma_{w''}} \\
 &= {\sum_{\substack{w'' \in N(v'') \\ \phi(w'') \ \mathrm{odd}}} (2 - 1) \Gamma_{v''}.\Gamma_{w''}} + {\sum_{\substack{w'' \in N(v'') \\ \phi(w'') \ \mathrm{even}}} (1 - 1) \Gamma_{v''}.\Gamma_{w''}} \ \ \ \ (\mathrm{by \ Lemma~\ref{evenprelim}(b)}) \\
 &= {\sum_{\substack{w'' \in N(v'') \\ \phi(w'') \ \mathrm{odd}}} 1} \ \ \ \ (\mathrm{by \ Lemma~\ref{evenprelim}(h)}) \\
 &= {\sum_{\substack{w \in N(v) \\ w \ \mathrm{odd}}} \sum_{\substack{w'' \in N(v'') \\ \phi(w'') = w}} 1} \ \ \ \ (\mathrm{by \ Lemma~\ref{evenprelim}(c)}) \\
 &= {\sum_{\substack{w \in N(v) \\ w \ \mathrm{odd}}} 1} \ \ \ \ (\mathrm{by \ Lemma~\ref{evenprelim}(i) \ with \ u=w})\\
 &= \begin{cases} 
 	\displaystyle{1 + {\sum_{\substack{w \in C(v) \\ w \ \mathrm{odd}}} 1}} \ \ \ \ &\mathrm{if \ v \ has \ an \ odd \ parent} \\ 
 	\displaystyle{{\sum_{\substack{w \in C(v) \\ w \ \mathrm{odd}}} 1}} \ \ \ \ &\mathrm{otherwise} 
     \end{cases}  \\
  &= (l_v \bmod 2) + r_v \ \ \ \ (\mathrm{by \ Lemma~\ref{lmod2} \ and \ Lemma~\ref{oddeven} \ since \ v \ is \ even}). \qedhere
\end{align*}  
\end{proof}

\begin{lemma}\label{eventhird}
 Suppose $v \in V(T_B)$ is even. Then, $D_2(v) = r_v+2s_v$.
\end{lemma}
\begin{proof}
We break up the computation of $D_2(v)$ into two cases: 

  \begin{flushleft}$\textbf{\textup{Case 1}} \colon \mathbf{l_v = 0}$ \end{flushleft}  
  In this case, Lemma~\ref{evenprelim}(f) tells us that $\# \phi^{-1}(v) = 2$. Since $l'_v$ and $r_v$ are nonnegative,  $r_v = 0$. Then,
    
  \begin{align*}
  D_2(v) &= \sum_{v'' \in \phi^{-1}(v)} \sum_{w'' \in C(v'')} \Gamma_{v''}.\Gamma_{w''} \\
  &= {\sum_{\substack{w \in C(v) \\ w \ \mathrm{even}}}}  \sum_{v'' \in \phi^{-1}(v)} {\sum_{\substack{w'' \in C(v'') \\ \phi(w'') = w}} \Gamma_{v''}.\Gamma_{w''}} \\
  & \pushright{(\textup{since Lemma~\ref{evenprelim}(c,d) imply that $\phi(w'') \in C(v)$ and is even})} \\
  &= {\sum_{\substack{w \in C(v) \\ w \ \mathrm{even}}}} 2 \\
  & \pushright{(\textup{by Lemma~\ref{evenprelim}(l,m) since Lemma~\ref{invimgphi}(b) implies that}\ \# \phi^{-1}(w) \in \{1,2\} )} \\
  &= r_v + 2s_v \ \ \ \ (\mathrm{by \ Lemma~\ref{oddeven}\ since \ v \ is \ even \ and \ r_v = 0}).
  \end{align*}  

  \begin{flushleft}$\textbf{\textup{Case 2}} \colon \mathbf{l_v \neq 0}$ \end{flushleft}
  In this case, Lemma~\ref{evenprelim}(g) implies that $\# \phi^{-1}(v) = 1$. Let $\{ v'' \} = \phi^{-1}(v)$. Then,
    
  \begin{align*}
 D_2(v) &= \sum_{w'' \in C(v'')} \Gamma_{v''}.\Gamma_{w''} \\
 &= {\sum_{\substack{w'' \in C(v'') \\ \phi(w'') \ \mathrm{odd}}} 1} + {\sum_{\substack{w'' \in C(v'') \\ \phi(w'') \ \mathrm{even}}} \Gamma_{v''}.\Gamma_{w''}} \ \ \ \ (\mathrm{by \ Lemma~\ref{evenprelim}(h)}) \\
 &= {\sum_{\substack{w \in C(v) \\ w \ \mathrm{odd}}} \sum_{\substack{w'' \in C(v'') \\ \phi(w'') = w}} 1} + {\sum_{\substack{w \in C(v) \\ w \ \mathrm{even}}} \sum_{\substack{w'' \in C(v'') \\ \phi(w'') = w}} \Gamma_{v''}.\Gamma_{w''}} \ \ \ \ (\mathrm{by \ Lemma~\ref{evenprelim}(c)}) \\
 &= {\sum_{\substack{w \in C(v) \\ w \ \mathrm{odd}}} 1} + {\sum_{\substack{w \in C(v) \\ w \ \mathrm{even}}} 2} \\
 & \pushright{(\mathrm{by \ Lemma~\ref{evenprelim}(i),(k) \ with }\ u=w \ \mathrm{and \ Lemma~\ref{evenprelim}(j)})}\\
 &= r_v + 2s_v \ \ \ \ (\mathrm{by \ Lemma~\ref{oddeven} \ since \ v \ is \ even}). \qedhere
\end{align*}    
\end{proof}

\begin{lemma}\label{eventotal}
 Suppose $v \in V(T_B)$ is even. Then,
 \[ D(v) = (l_v \bmod 2)+2r_v+2s_v .\]
\end{lemma}
\begin{proof}
Combine Lemmas ~\ref{evenfirst}, ~\ref{evensecond} and ~\ref{eventhird}.
\end{proof}

\subsection{Computation of $D(v)$ for an odd vertex $v$}
Suppose $v \in V(T_B)$ is odd. Let $S_0(v),S_1(v),S_2(v)$ denote the partition of $\phi^{-1}(v)$ constructed in Lemma~\ref{invimgphi}(c). 

\begin{lemma}\label{oddprelim}
Suppose $v \in V(T_B)$ is odd. Let $v'' \in S_0(v), w'' \in N(v''),v' = \phi_1(v'')$ and $w = \phi(w'')$. 
 \begin{enumerate}[\upshape (a)]
  \item The component $\Gamma_{v'}$ is the strict transform of $\Gamma_v$ in $Y$ and $v'$ is odd. The image $\psi(\Gamma_{v''})$ is not a point.
  \item We have that
  \[ \{ w'' \in C(v'') \ | \ m_{w''} = 2 \} = S_1(v) .\] 
  We also have that $\# S_1(v) = s_v$.
  \item If $v'' \in C(w'')$ and $m_{w''} = 2$, then $w = p_v$ and $w$ is odd.
  \item If $p_v$ is odd, there exists a unique $u'' \in \phi^{-1}(p_v)$ such that $v'' \in C(u'')$.
  \item The map $\phi$ induces a bijection between the sets $\{ w'' \in C(v'') \setminus S_2(v)\ |\ m_{w''} = 1 \}$ and $\{ w \in C(v)\ | \ w \ \textup{is even} \}$.
  \item We have that $\Gamma_{v''}.\Gamma_{w''} = 1$.
 \end{enumerate}
\end{lemma}
\begin{proof} \hfill
 \begin{enumerate}[\upshape (a)]
  \item Since $v'' \in S_0(v)$ and $v' = \phi_1(v'')$, it follows from Lemma~\ref{invimgphi}(c)(ii) that $\Gamma_{v'}$ is the strict transform of $\Gamma_v$ in $Y$. Since $\phi_1(v'') = v'$, it follows that $\psi(\Gamma_{v''}) = \psi_2(\Gamma_{v'}) = \Gamma_v$. Therefore $\psi(\Gamma_{v''})$ is not a point. Lemma~\ref{invimgphi}(c)(ii) also implies that $v'$ is odd. 
 
  \item Suppose $w'' \in C(v'')$ and $m_{w''} = 2$. Let $w' = \phi_1(w'')$. Since $w'' \in C(v'')$, Lemma~\ref{adjacency}(a) implies that $w' \in C(v')$. Since odd components of $Y$ do not intersect and (a) implies that $v'$ is odd, $w'$ is even. Since $m_{w''} = 2$, Lemma~\ref{mult2} tells us that $w$ is odd and $w'$ is not an even leaf of $T_Y$. Let $T_0,T_1,T_2$ be the partition of $\phi_2^{-1}(w)$ as in Lemma~\ref{invimgphi}(c). Since $w'$ is even, Lemma~\ref{oddY} tells us that $w' \notin T_0$. Since $w'$ is not an even leaf of $T_Y$, the displayed equation in the proof of Lemma~\ref{mult2} shows that $w' \in T_1$. Since $w' \in T_1$, Lemma~\ref{invimgphi}(c)(vii) shows that $p_{w'} \in T_0$. Since $w' \in C(v')$, it follows that $v' = p_{w'} \in T_0$ and therefore $\phi_2(v') \in \phi_2(T_0) = \{w\}$, which implies that $v = w$. Finally, $w'' \in \phi_1^{-1}(w') \subseteq \phi_1^{-1}(T_1) = S_1(v)$.
  
  Conversely, suppose $w'' \in S_1(v)$. Since $v'' \in S_0(v)$, Lemma~\ref{invimgphi}(c)(i,vii) show that $w'' \in C(v'')$ and $m_{w''} = 2$. Lemma~\ref{invimgphi}(c)(iii) implies that $\# S_1(v) = s_v$.
  
  \item Suppose $v'' \in C(w'')$ and $m_{w''} = 2$. Since $v'' \in C(w'')$ and $\psi(\Gamma_{v''})$ is not a point by (a), Lemma~\ref{adjacency}(b) tells us that $v \in C(w)$. Since $m_{w''} = 2$, Lemma~\ref{mult2} tells us that $w$ is odd. 
  
  \item Suppose $p_v$ is odd. Let $u = p_v$. Let $T_0,T_1,T_2$ be the partition of $\phi_2^{-1}(u)$ as in Lemma~\ref{invimgphi}(c). Let $u' \in T_1$ be the unique vertex such that $\psi_2(\Gamma_{u'}) = \Gamma_u \cap \Gamma_v$. Since (a) implies that $\Gamma_{v'}$ is the strict transform of $\Gamma_v$ in $Y$, the proof of Lemma~\ref{invimgphi}(c)(iii) in the case of the odd vertex $u$ shows that $v' \in C(u')$. Lemma~\ref{invimgphi}(c) applied to the odd vertex $u$ tells us that $\phi_1$ induces a bijection between $\phi^{-1}(u)$ and $\phi_2^{-1}(u)$. This shows that there exists a unique $u'' \in V(T_X)$ such that $\phi_1(u'') = u'$. Since $v'$ is odd by (a), Lemma~\ref{invimgphi}(a) and Lemma~\ref{branchlocus}(a) then imply that $\psi_1^{-1}(\Gamma_{v'}) = \Gamma_{v''}$. Since $\phi_1^{-1}(u') = \{ u'' \}$, it follows that $\psi_1^{-1}(\Gamma_{u'}) = \Gamma_{u''}$. Therefore, $\Gamma_{u''} \cap \Gamma_{v''} = \psi_1^{-1}(\Gamma_{u'} \cap \Gamma_{v'}) \neq \emptyset$. This implies that either $u'' \in C(v'')
$, or $v'' \in C(u'')$. Since $v' \in C(u')$, Lemma~\ref{adjacency}(a) implies that $v'' \in C(u'')$. This proves the existence of $u''$.
  
  Suppose $u'' \in \phi^{-1}(u)$ be such that such that $v'' \in C(u'')$. Then, Lemma~\ref{adjacency}(a) implies that $\phi_1(u'') = p_{v'}$. Since $v'$ is odd (by (a)) and $\Gamma_{p_{v'}}$ intersects $\Gamma_{v'}$, Lemma~\ref{branchlocus}(a) and Lemma~\ref{invimgphi}(a) imply that $\# \phi_1^{-1}(p_{v'}) = 1$. This proves uniqueness of $u'' \in \phi^{-1}(u)$ such that $v'' \in C(u'')$.
  
  \item Suppose $w'' \in C(v'') \setminus S_2(v)$ and $m_{w''} = 1$. We will first show $\psi(\Gamma_{w''})$ is not a point. Suppose $\psi(\Gamma_{w''})$ is a point. Since $w'' \in C(v'')$, Lemma~\ref{adjacency}(b) implies that $w = \phi(w'') = \phi(v'') = v$. Since $m_{w''} = 1$, Lemma~\ref{invimgphi}(c)(i,ii,iii) then imply that $w'' \in S_2(v)$, which is a contradiction. Therefore, $\psi(\Gamma_{w''})$ is not a point. Lemma~\ref{adjacency}(a) then implies that $w \in C(v)$.
  
  Suppose $w$ is odd. Let $w' = \phi_1(w'')$. Since $\psi(\Gamma_{w''})$ is not a point,  $w'' \in S_0(w)$. Part (a) applied to $w''$ implies that $w'$ is odd. Part (a) implies that $v'$ is odd.  Since $w'' \in C(v'')$, Lemma~\ref{adjacency}(a) implies that $w' \in C(v')$. This is a contradiction since odd components of $Y$ cannot intersect. Therefore $w$ is even. This shows one inclusion.
    
  Now suppose $u \in C(v)$ is even. Let $u' \in V(T_Y)$ be the vertex corresponding to the strict transform of $\Gamma_u$ in $Y$. Part (a) implies that $v'$ is the vertex corresponding to the strict transform of $\Gamma_v$ and $v'$ is odd. Lemma~\ref{oddY} implies that $u'$ is even. This in turn implies that $\psi_2$ is an isomorphism above a neighbourhood of $\Gamma_{u}$, and therefore $u' \in C(v')$. Since $v'$ is odd and $u' \in C(v')$, Lemma~\ref{invimgphi}(b) applied to $u$ implies that $\# \phi^{-1}(u) = 1$. Let $\phi^{-1}(u) = \phi_1^{-1}(u') = \{ u'' \}$. Since $\psi_1^{-1}(\Gamma_{v'}) = \Gamma_{v''}$ and $\psi_1^{-1}(\Gamma_{u'}) = \Gamma_{u''}$, it follows that $\Gamma_{v''} \cap \Gamma_{u''} = \psi_1^{-1}(\Gamma_{v'} \cap \Gamma_{u'})$ is not empty. In particular,  $u'' \in N(v'')$. Since $\phi_1(u'') = u' \in C(v') = C(\phi_1(v''))$, Lemma~\ref{adjacency}(a) implies that $u'' \in C(v'')$. This shows the opposite inclusion.
  
  \item Since $\phi(v'') = v$ is odd, Lemma~\ref{intersect} tells us that $ \Gamma_{v''}.\Gamma_{w''} < 2$. On the other hand, since $w'' \in N(v'')$, we have that $\Gamma_{v''}.\Gamma_{w''} \geq 1$. \qedhere
 \end{enumerate}
\end{proof}

We will now compute $\sum_{v'' \in S_i(v)} \delta(v'')$ for each $i \in \{0,1,2\}$, in terms of $l_v,r_v$ and $s_v$.

\begin{lemma}\label{odd0}
Suppose $v \in V(T_B)$ is odd. Then
 \[ \sum_{v'' \in S_0(v)} \delta(v'') = \left\{ \begin{array}{ll}
-2+l_v+2s_v \ &\mathrm{if}\ p_v\ \mathrm{is}\ \mathrm{even} \\                                                                                                                                                                                                                          
-1+l_v+2s_v \ &\mathrm{if}\ p_v\ \mathrm{is}\ \mathrm{odd}.                                                                                                                                                                                                                        \end{array}
 \right. \] 
\end{lemma}
\begin{proof}
Let $S_0 = S_0(v), S_1 = S_1(v)$ and $S_2 = S_2(v)$. Lemma~\ref{invimgphi}(c)(ii) implies that $\# S_0 = 1$. Let $\tilde{v} \in S_0$. Since $S_0$ consists of a single vertex $\tilde{v}$, 
\[ \sum_{v'' \in S_0} \delta(v'') = \delta(\tilde{v}) = (1-m_{\tilde{v}})\ \chi(\Gamma_{\tilde{v}}) + \sum_{w'' \in N({\tilde{v}})} (m_{w''} - 1) \Gamma_{\tilde{v}}.\Gamma_{w''} + \sum_{w'' \in C({\tilde{v}})} \Gamma_{\tilde{v}}.\Gamma_{w''}  .\]
We will compute each of the three terms in this sum separately.

By Lemma~\ref{invimgphi}(c)(ii), 
\[ (1-m_{\tilde{v}})\ \chi(\Gamma_{\tilde{v}}) = (1-m_{\tilde{v}})\ \chi(\P^1_k) =  (1-2)(2) = -2 .\] 

Now 
\begin{align*}
 \sum_{w'' \in N({\tilde{v}})} (m_{w''} - 1) \Gamma_{\tilde{v}}.\Gamma_{w''}  
 &= \sum_{w'' \in N({\tilde{v}})} (m_{w''} - 1) \ \ \ \ (\mathrm{by \ Lemma~\ref{oddprelim}(f)}) \\
 &= \sum_{w'' \in S_1} (2 - 1) + \sum_{w'' \in C({\tilde{v}}) \setminus S_1} (1 - 1) + {\sum_{\substack{w'' \in V(T_X) \\ \tilde{v} \in C(w'') }}}  (m_{w''} - 1) \\
 & \pushright{(\mathrm{by \ Lemma~\ref{oddprelim}(b)})} \\
 &= s_v + {\sum_{\substack{w'' \in V(T_X) \\ \tilde{v} \in C(w'') }}} (m_{w''} - 1) \ \ \ \ (\mathrm{by \ Lemma~\ref{oddprelim}(b)}) \\
 &= s_v + {\sum_{\substack{w'' \in \phi^{-1}(p_v) \\ \tilde{v} \in C(w'') \\ \phi(w'') \ \mathrm{is \ odd} }}} (m_{w''} - 1) \ \ \ \ (\mathrm{by \ Lemma~\ref{oddprelim}(c)}) \\
 &= {\begin{cases}
      s_v &\textup{if $p_v$ is even} \\
      s_v+1 &\textup{if $p_v$ is odd}
     \end{cases}
}\ \ \ \ (\mathrm{by \ Lemma~\ref{oddprelim}(d)}). 
\end{align*}

Now
\begin{align*}
 \sum_{w'' \in C({\tilde{v}})} \Gamma_{\tilde{v}}.\Gamma_{w''} 
 &= \sum_{w'' \in C({\tilde{v}})} 1 \ \ \ \ (\mathrm{by \ Lemma~\ref{oddprelim}(f)}) \\
 &= {\sum_{\substack{w'' \in C({\tilde{v}}) \\ m_{w''} = 2 }}} 1 + {\sum_{w'' \in S_2}} 1 + {\sum_{\substack{w'' \in C({\tilde{v}}) \setminus S_2 \\ m_{w''} = 1  }}} 1 \\
 & \pushright{(\textup{by Lemmas~\ref{construction}(b),~\ref{invimgphi}(c)(i,iv,vii)})} \\
 &= s_v+l'_v+{\sum_{\substack{w'' \in C({\tilde{v}}) \setminus S_2 \\ m_{w''} = 1  }}} 1 \ \ \ \ \ \ \ \ \ (\mathrm{by \ Lemma~\ref{oddprelim}(b) \ and \ Lemma~\ref{invimgphi}(c)(iv)}) \\
 &= s_v+l'_v+r_v \ \ \ \ (\textup{by Lemma~\ref{oddeven} since $v$ is odd,  and by  Lemma~\ref{oddprelim}(e)}) \\
 &= s_v+l_v. 
\end{align*}

Adding the three previous equalities gives us
\[ \sum_{v'' \in S_0(v)} \delta(v'') =  \delta(\tilde{v})  =  \left\{ \begin{array}{ll}
-2+l_v+2s_v  &\mathrm{if}\ p_v\ \mathrm{is}\ \mathrm{even} \\                                                                                                                                                                                                                          
-1+l_v+2s_v  &\mathrm{if}\ p_v\ \mathrm{is}\ \mathrm{odd}.  \end{array} \right. \qedhere                                                                                                                                                                                                                        \] 
\end{proof}

\begin{lemma}\label{odd1}
Suppose $v \in V(T_B)$ is odd. Then
 \[ \sum_{v'' \in S_1(v)} \delta(v'') = s_v .\]
\end{lemma}
\begin{proof}
Let $S_1 = S_1(v)$. Let $\tilde{v}$ be the unique element of $S_0(v)$. Suppose $v'' \in S_1$. Lemma~\ref{invimgphi}(c)(iii,viii) tells us that $\Gamma_{v''} \isom \P^1_k$, $v'' \in C(\tilde{v})$, $m_{v''} = 2$ and $\psi(\Gamma_{v''}) = \Gamma_v \cap \Gamma_u$ for an odd $u \in C(v)$. 

Since $\psi(\Gamma_{v''})$ is a point that belongs to two odd components of $(\mathrm{Bl}_n(\P^1_R))_s$, Lemma~\ref{neigh}(a)(i,ii) tell us that $\# N(v'') = 2$ and $\# C(v'') = 1$. Suppose $w'' \in N(v'')$. Lemma~\ref{neigh}(a)(iii,iv) tell us that $\phi(w'')$ is odd and $m_{w''} = 2$. 
Since $\phi(w'')$ is odd, Lemma~\ref{intersect} tells us that $\Gamma_{v''}.\Gamma_{w''} < 2$. On the other hand, since $w'' \in N(v'')$, we have that $\Gamma_{v''}.\Gamma_{w''} \geq 1$. This implies that
\begin{align*}
 \delta(v'') &= (1-m_{v''})\ \chi(\Gamma_{v''}) + \sum_{w'' \in N({v''})} (m_{w''} - 1) \Gamma_{v''}.\Gamma_{w''} + \sum_{w'' \in C({v''})} \Gamma_{v''}.\Gamma_{w''} \\
 &= (1-2)2 + (2-1)1 + (2-1)1 + 1 \\
 &= 1 .
\end{align*}
Therefore
\[ \sum_{v'' \in S_1(v)} \delta(v'') = \sum_{v'' \in S_1(v)} 1 = s_v \ \ \ \ ({\textup{since  Lemma~\ref{invimgphi}(c)(iii) implies that $\# S_1 = s_v$}}). \qedhere \] 
\end{proof}

\begin{lemma}\label{odd2}
Suppose $v \in V(T_B)$ is odd. Then
 \[ \sum_{v'' \in S_2(v)} \delta(v'') = l_v-r_v .\]
\end{lemma}
\begin{proof}
Let $S_2 = S_2(v)$ and $S_0(v) = \{ \tilde{v} \}$. Suppose $v'' \in S_2$. Lemma~\ref{invimgphi}(c)(iv,viii) tells us that $\Gamma_{v''} \isom \P^1_k$, $v'' \in C(\tilde{v})$, $m_{v''} = 1$ and $\psi(\Gamma_{v''}) = \Gamma_v \cap H$ where $H$ is an irreducible horizontal divisor occuring in $(f)$ on $\mathrm{Bl}_n(\P^1_R)$. 

Since $\psi(\Gamma_{v''})$ is a point that belongs to a unique odd component of $(\mathrm{Bl}_n(\P^1_R))_s$, Lemma~\ref{neigh}(a)(i,ii) tell us that $\# N(v'') = 1$ and $\# C(v'') = 0$. Since $v'' \in C(\tilde{v})$, we have that $N(v'') = \{ \tilde{v} \}$. Lemma~\ref{invimgphi}(c)(ii) implies that $m_{\tilde{v}} = 2$. Since $\tilde{v} \in N(v'')$ and $\phi(\tilde{v})$ ($=v$) is odd, Lemma~\ref{intersect} applied to the pair $v'',\tilde{v}$ tells us that $\Gamma_{v''}.\Gamma_{w''} < 2$. On the other hand, since $\tilde{v} \in N(v'')$, we have that $\Gamma_{v''}.\Gamma_{w''} \geq 1$. This implies that
\begin{align*} 
\delta(v'') &= (1-m_{v''})\ \chi(\Gamma_{v''}) + \sum_{w'' \in N({v''})} (m_{w''} - 1) \Gamma_{v''}.\Gamma_{w''} + \sum_{w'' \in C({v''})} \Gamma_{v''}.\Gamma_{w''} \\
&= (1-1)2 + (2-1)1 + 0 \\
&= 1 .
\end{align*}
Therefore
\[ \sum_{v'' \in S_2(v)} \delta(v'') = \sum_{v'' \in S_2(v)} 1 = l'_v = l_v-r_v \ \ \ \ ({\textup{since  Lemma~\ref{invimgphi}(c)(iv) implies that $\# S_2 = l'_v$}}) . \qedhere \] 
\end{proof}

\begin{lemma}\label{oddtotal}
 Suppose $v \in V(T_B)$ is odd (in particular, $v$ is not the root). Then 
 \[ D(v) = \left\{ 
 \begin{array}{ll}
 -2-r_v+3s_v+2l_v \ &\mathrm{if}\ v \ \mathrm{is\ odd \ and}\ p_v\ \mathrm{is\ even} \\                                                                                                                                                                                                                          
 -1-r_v+3s_v+2l_v \ &\mathrm{if}\ v \ \mathrm{is\ odd \ and}\ p_v\ \mathrm{is\ odd}  .
 \end{array}
 \right.  \]
\end{lemma}
\begin{proof}
 Combine Lemmas~\ref{odd0},\ref{odd1},\ref{odd2}.
\end{proof}

\subsection{Formula for $D(v)$}
\begin{thm}\label{localDel}
 Let $v \in V(T_B)$. Then 
 \[ D(v) = \left\{ 
 \begin{array}{ll}
 (l_v \bmod 2)+2r_v+2s_v \ &\mathrm{if}\  v\ \mathrm{is \ even} \\
 -2-r_v+3s_v+2l_v \ &\mathrm{if}\ v \ \mathrm{is\ odd \ and}\ p_v\ \mathrm{is\ even} \\                                                                                                                                                                                                                          
 -1-r_v+3s_v+2l_v \ &\mathrm{if}\ v \ \mathrm{is\ odd \ and}\ p_v\ \mathrm{is\ odd}  .
 \end{array}
 \right. \] 
\end{thm}
\begin{proof}
 This follows directly from Lemma~\ref{eventotal} and Lemma~\ref{oddtotal}.
\end{proof}

\section{Comparison of the two discriminants}\label{comparison}
One might hope that the inequality $D(v) \leq d(v)$ holds for every vertex $v \in V(T_B)$, but this is not true. It is however true after a slight alteration of the function $D$. 

\subsection{A new break-up of the Deligne discriminant}
Define a new function $E$ on $V(T_B)$ as follows:
\[ E(v) = \left\{ { \begin{array}{ll}
            \displaystyle{- (l_v \bmod 2) - \sum_{\substack{v' \in C(v) \\ v' \ \mathrm{odd}}} \left( 2-  \wt_{v'}(\wt_{v'}-1)  \right)} \ &\textup{if $v$ is even} \\
            \displaystyle{r_v + s_v + 2 -  \wt_v(\wt_v-1)  - \sum_{\substack{v' \in C(v) \\ v' \ \mathrm{odd}}} \left( 2-  \wt_{v'}(\wt_{v'}-1) \right)} \ &\textup{if $v$ is odd, $p_v$ even} \\
            \displaystyle{r_v + s_v + 1 -  \wt_v(\wt_v-1)  - \sum_{\substack{v' \in C(v) \\ v' \ \mathrm{odd}}} \left( 2-  \wt_{v'}(\wt_{v'}-1) \right)} \ &\textup{if $v$ and $p_v$ are odd.}
           \end{array} } \right.
 \]
 For $v \in V(T_B)$, set $D'(v) := D(v) +E(v)$.
 
 Using Lemma~\ref{oddeven}, we get  
 \[ \sum_{\substack{v' \in C(v) \\ v' \ \textrm{odd}}} 2 = \begin{cases}2s_v \ \textup{if $v$ is odd} \\ 2r_v \ \textup{if $v$ is even} \end{cases} .\]
 We can use this, along with Theorem~\ref{localDel} to simplify the expression of $D'$. 
 \begin{equation}\label{formulaforD'} D'(v) = \left\{ { \begin{array}{ll}
            2s_v + \sum_{\substack{v' \in C(v) \\ v' \ \textrm{odd}}}  \wt_{v'}(\wt_{v'}-1)   \ &\textup{if $v$ is even} \\
            2 \big(l_v+s_v \big) - \wt_v(\wt_v-1)  + \sum_{\substack{v' \in C(v) \\ v' \ \textrm{odd}}} \wt_{v'}(\wt_{v'}-1)   \ &\textup{if $v$ is odd}
           \end{array} } \right .
 \end{equation}
 
 \begin{lemma}\label{breakupeq} The following equalities hold.
  \[ \sum_{\substack{v \in V(T_B) \\ v \ \mathrm{even}}} \sum_{\substack{v' \in C(v) \\ v' \ \mathrm{odd}}} -\left( 2-  \wt_{v'}(\wt_{v'}-1)  \right)  + \sum_{\substack{v \in V(T_B) \\ v \ \mathrm{odd}}} \left( 2 -  \wt_v(\wt_v-1)  - \sum_{\substack{v' \in C(v) \\ v' \ \mathrm{odd}}} \left( 2-  \wt_{v'}(\wt_{v'}-1) \right) \right)  = 0 .\]
  \[ \sum_{\substack{v \in V(T_B) \\ v \ \mathrm{even}}} -(l_v \bmod 2) + \sum_{\substack{v \in V(T_B) \\ v \ \mathrm{odd}}} r_v = 0  .\]
  \[ \sum_{\substack{v \in V(T_B) \\ v \ \mathrm{odd} \\ \textup{$p_v$ is odd}}} -1 + \sum_{\substack{v \in V(T_B) \\ v \ \mathrm{odd}}} s_v = 0  .\] 
 \end{lemma}
\begin{proof} The first equality can be rewritten as
 \[ \sum_{v \in V(T_B)} \sum_{\substack{v' \in C(v) \\ v' \ \mathrm{odd}}} -\left( 2-  \wt_{v'}(\wt_{v'}-1)  \right) + \sum_{\substack{v \in V(T_B) \\ v \ \mathrm{odd}}} \left( 2 -  \wt_v(\wt_v-1)  \right)  = 0 .\]
Since the root is an even vertex, every odd vertex has a parent. This implies that 
\[ \sum_{v \in V(T_B)}  \sum_{\substack{v' \in C(v) \\ v' \ \mathrm{odd}}} -\left( 2-  \wt_{v'}(\wt_{v'}-1)  \right) = - \sum_{\substack{v \in V(T_B) \\ v \ \mathrm{odd}}} \left( 2 -  \wt_v(\wt_v-1)  \right)   .\]
 We have that
\begin{align*}
 \sum_{\substack{v \in V(T_B) \\ v \ \mathrm{even}}} -(l_v \bmod 2) &= \sum_{\substack{v \in V(T_B) \\ v \ \mathrm{even} \\ v \ \mathrm{has \ an \ odd \ parent}}} -1 \ \ \ \ (\mathrm{by\ Lemma~\ref{lmod2}}) \\
 &= \sum_{\substack{w \in V(T_B) \\ w \ \mathrm{odd}}} \sum_{\substack{v \in C(w) \\ w \ \mathrm{even}}}  -1  \\
 &= \sum_{\substack{w \in V(T_B) \\ w \ \mathrm{odd}}} -r_w \ \ \ \ (\mathrm{by\ Lemma~\ref{oddeven}}) .
\end{align*}
 We have that
\begin{align*}
 \sum_{\substack{v \in V(T_B) \\ v \ \mathrm{odd} \\ \mathrm{p_v\ is\ odd}}} -1 &= \sum_{\substack{w \in V(T_B) \\ w \ \mathrm{odd}}} \sum_{\substack{v \in C(w) \\ w \ \mathrm{odd}}}  -1  \\
 &= \sum_{\substack{w \in V(T_B) \\ w \ \mathrm{odd}}} -s_w \ \ \ \ (\mathrm{by\ Lemma~\ref{oddeven}}) . 
\end{align*} \qedhere
\end{proof}

\begin{lemma}\label{newsum}
\[ \sum_{v \in V(T_B)} E(v) = 0 .\] 
\end{lemma}
\begin{proof}
The sum of the left hand sides of the three equalities in Lemma~\ref{breakupeq} equals  $\sum_{v \in V(T_B)} E(v)$, which is therefore $0$.
\end{proof}

For an odd $v \in V(T_B)$ such that $\wt_v > 2$, let $L_v = \{ w \in C(v) \ | \ \wt_{w} = 2  \}$. 
Define a new function $D''$ on $V(T_B)$ as follows:
\[ D''(v) = \left\{ 
\begin{array}{ll}
D'(v)-2 \ &\textup{if $v$ is an odd leaf and } \wt_v = 2 \\
D'(v) \ &\textup{if  $v$ is odd, not a leaf, and } \wt_v = 2 \\
D'(v)+2\# L_v \ &\textup{if $v$ is odd, and $\wt_v > 2$} \\
D'(v) \ &\textup{if $v$ is even}.
\end{array}
\right. \]

\begin{lemma}\label{newnewsum}
\[ \sum_{v \in V(T_B)} D''(v) = \sum_{v \in V(T_B)} D'(v) .\]  
\end{lemma}
\begin{proof}
For an odd leaf $v \in V(T_B)$ such that $\wt_v = 2$, let $q_v$ denote the least ancestor of $v$ such that $\wt_{q(v)} \geq 3$ (here least ancestor means the ancestor farthest away from the root); such an ancestor exists as the root has weight $2g+2 \geq 3$. If $v \in V(T_B)$ is odd and $\wt_v = 2$, then $p_v$ must also be odd by Lemma~\ref{oddeven}. A repeated application of this fact tells us that if $v$ is an odd leaf such that $\wt_v = 2$, then $q_v$ is odd. 

For any vertex $v \in V(T_B)$, let $T_v$ denote the complete subtree of $T_B$ with root $v$ (see section 8 for the definition of complete subtree). Suppose $v$ is an odd vertex such that $\wt_v > 2$. We will now prove the following three claims.
\begin{itemize}
 \item If $w \in L_v$ and $u \in T_w$, then $u$ is odd and $\wt_u = 2$.
 \item If $w \in L_v$, then $T_w$ is a chain (that is, every vertex in $T_w$ has at most one child). 
 \item If $v' \in V(T_B)$ is an odd leaf such that $\wt_{v'} = 2$ and $q_{v'} = v$, then there exists a unique $w \in L_v$ such that $v' \in V(T_w)$. 
\end{itemize}

Suppose $w \in L_v$ and $u \in T_w$. Since $u \in T_w$, the definition of the function $\wt$ tells us that $\wt_u \leq \wt_w = 2$. On the other hand, Lemma~\ref{weight} tells us that $\wt_u \geq 2$. Therefore, $\wt_u = 2$. A repeated application of Lemma~\ref{oddeven} along the path from $v$ to $u$ tells us that $u$ is odd. This proves the first claim.

Suppose $w \in L_v$ and $u \in T_w$. Suppose $u_1,u_2 \in C(u)$ are distinct. The first claim shows $\wt_{u_1} = \wt_{u_2} = 2$. The definition of $\wt$ then tells us that $\wt_w \geq \wt_u \geq \wt_{u_1}+\wt_{u_2}$. Since $\wt_w = 2$ and $\wt_{u_1}+\wt_{u_2} = 4$, this is a contradiction. Therefore every vertex in $T_v$ has at most one child, and this proves the second claim.

Suppose $v' \in V(T_B)$ is an odd leaf such that $\wt_{v'} = 2$ and $q_{v'} = v$. Let $w$ be the greatest ancestor of $v'$ such that $\wt_w = 2$ (here greatest ancestor means the ancestor closest to the root). Then, $\wt_{p_w} > 2$. The definition of $q$ then implies $p_w = q_{v'} = v$. This implies that $w \in L_v$. If $w_1,w_2 \in L_v$, then $T_{w_1}$ and $T_{w_2}$ have no vertices in common. This proves that every $v' \in M_v$ can belong to $V(T_w)$ for at most one $w \in L_v$. This finishes the proof of the third claim.

Let $M_v = \{ v' \in V(T_B) \ | \ \textup{$v'$ is an odd leaf},\ \wt_{v'} = 2,\ q_{v'} = v  \}$. We will now use the claims above to show that there is a bijection $\kappa \colon L_v \rightarrow M_v$. Let $w \in L_v$. Let $v'$ be the unique leaf in the chain $T_w$. Then $v'$ is an odd leaf and $\wt_{v'} = 2$. Furthermore, $w$ is an ancestor of $v'$ such that $\wt_{w} = 2$ and $\wt_v= \wt_{p_w} > 2$, which shows $q_{v'} = v$. Set $\kappa(w) = v'$. The third claim shows that $\kappa$ is a bijection. Therefore $\# M_v = \# L_v$.

This implies that
\[ \sum_{\substack{v' \ \mathrm{is \ an \ odd \ leaf} \\ \wt_{v'} = 2}} 2 = \sum_{\substack{v \ \mathrm{odd} \\ \wt_v > 2}} \sum_{v' \in M_v} 2 = \sum_{\substack{v \ \mathrm{odd} \\ \wt_v > 2}} 2 \ \#M_v = \sum_{\substack{v \ \mathrm{odd} \\ \wt_v > 2}} 2 \#L_v .\]
This tells us that
\[ \sum_{v \in V(T_B)}( D''(v) - D'(v) ) = \sum_{\substack{v \in V(T_B) \\ v \ \mathrm{odd \ leaf} \\ \wt_v = 2}} -2 + \sum_{\substack{v \in V(T_B) \\ v \ \mathrm{odd} \\ \wt_v > 2}} 2\# L_v = 0. \qedhere \]
\end{proof}

\begin{lemma}\label{weightls} \hfill
\begin{enumerate}[\upshape (a)]
 \item If $v \in V(T_B)$, then \[ \wt_v \geq l'_v + 3r_v + 2s_v \geq l_v+2s_v.\]
 \item If $r_v = s_v = 0$, then $\wt_v = l'_v$.
\end{enumerate} 
\end{lemma}
\begin{proof} \hfill
\begin{enumerate}
 \item Suppose $u \in C(v)$. Lemma~\ref{weight} tells us that $\wt_u \geq 2$. If $u$ is of odd weight, then $\wt_u \geq 3$.   Therefore
\begin{align*}
 \wt_v &= l'_v + \sum_{u \in C(v)} \wt_u \ \ \ \ (\textup{by the definitions of $l'_v$ and $\wt$}) \\
 &\geq l'_v + {\sum_{\substack{u \in C(v) \\ \wt_u \ \textup{is odd}}}} 3 + {\sum_{\substack{u \in C(v) \\ \wt_u \ \textup{is even}}}} 2 \\
 &\geq l'_v + 3r_v + 2s_v \\
 &= l_v+2r_v+2s_v \\
 &\geq l_v+2s_v. 
\end{align*}
\item If $r_v = s_v = 0$, then $C(v) = \emptyset$ and therefore $\wt_v = l'_v + \sum_{u \in C(v)} \wt_u = l'(v)$. \qedhere
\end{enumerate}
\end{proof}

We are now ready to compare the two discriminants. We first compare the local contributions.
\begin{lemma}\label{localinequality}
If $v \in V(T_B)$, then $D''(v) \leq d(v)$. If $v$ is even, then $D''(v) = d(v)$ if and only if every even child of $v$ has weight $2$. If $v$ is odd, then $D''(v) = d(v)$ if and only if either $\wt_v = 2$ or $\wt_v = 3$ and $v$ has no even children.
\end{lemma}
\begin{proof}

If $v \in V(T_B)$ is even, then
\begin{align*}
D''(v) - d(v) &= D'(v) - d(v) \\ 
&= 2s_v + \sum_{\substack{v' \in C(v) \\ v' \ \textrm{odd}}} \wt_{v'}(\wt_{v'}-1) - \sum_{v' \in C(v)} \wt_{v'}(\wt_{v'}-1) \\
& \pushright{(\textup{by Lemma~\ref{formulaford} and Equation~\ref{formulaforD'} })} \\
&= \sum_{\substack{v' \in C(v) \\ v' \ \textrm{even}}} \left( 2 - \wt_{v'}(\wt_{v'}-1)  \right) \ \ \ \ \ \ \ \ \ (\mathrm{by \ Lemma \ } {\ref{oddeven}}) \\
&\leq 0 \ \ \ \ (\textup{by Lemma~\ref{weight} }).
\end{align*}

From this, it follows that if $v$ is even, then $D''(v) = d(v)$ if and only if the inequality above is actually an equality, that is, if and only if every even child of $v$ has weight $2$.

From now on assume $v \in V(T_B)$ is odd. Then
\begin{align}\label{D'minusd}
D'(v) - d(v) 
&= 2 \big(l_v+s_v \big) -  \wt_v(\wt_v-1) + \sum_{\substack{v' \in C(v) \\ v' \ \textrm{odd}}} \wt_{v'}(\wt_{v'}-1) - \sum_{v' \in C(v)} \wt_{v'}(\wt_{v'}-1) \\
&= 2 \big(l_v+s_v \big) -  \wt_v(\wt_v-1) - \sum_{\substack{v' \in C(v) \\ v' \ \textrm{even}}} \wt_{v'}(\wt_{v'}-1) \nonumber,  
\end{align}
where the first equality follows from Lemma~\ref{formulaford} and Equation~\ref{formulaforD'}. Lemma~\ref{weight} tells us that $\wt_v \geq 2$. We will handle vertices with $\wt_v = 2$ and with $\wt_v \geq 3$ separately.

Suppose $\wt_v = 2$. Lemma~\ref{weightls}(a) implies that $l'_v + 3r_v + 2s_v \leq \wt_v = 2 $. This implies that $r_v = 0$. Lemma~\ref{weightls}(b) implies that either 
\begin{enumerate}[(i)]
 \item $l'_v = 2$ and $s_v = 0$, or,
 \item $l'_v = 0$ and $s_v = 1$.
\end{enumerate}
In both cases, since $r_v = 0$ and $v$ is odd, Lemma~\ref{oddeven} tells us that 
\[ \sum_{\substack{v' \in C(v) \\ v' \ \textrm{even}}} \wt_{v'}(\wt_{v'}-1) = 0 .\]
In case (i), we have that $v$ is an odd leaf of weight $2$ and 
\begin{align*}
 D''(v) - d(v) &= D'(v) - d(v) - 2 \\
 &= 2 \big(l_v+s_v \big) - \wt_v(\wt_v-1) - \sum_{\substack{v' \in C(v) \\ v' \ \textrm{even}}} \wt_{v'}(\wt_{v'}-1) - 2 \\
 &= 2(2+0)-2(2-1)+0-2 \\
 &= 0. 
\end{align*}
In case (ii), we have that $v$ is not a leaf and $\wt_v = 2$ and 
\begin{align*}
D''(v) - d(v) &= D'(v)-d(v) \\
&= 2 \big(l_v+s_v \big) - \wt_v(\wt_v-1) - \sum_{\substack{v' \in C(v) \\ v' \ \textrm{even}}} \wt_{v'}(\wt_{v'}-1) \\
&= 2(0+1)-2(2-1)-0 \\
&= 0. 
\end{align*}

Now suppose $\wt_v \geq 3$. By definition, $\# L_v \leq s_v$. 
\begin{align}\label{weight3}
 2\# L_v + 2 \big(l_v+s_v \big) - \wt_v(\wt_v-1)  &\leq 2 \big(l_v+2s_v \big) -  \wt_v(\wt_v-1)  \\
 &\leq 2\wt_v-\wt_v(\wt_v-1) \ \ \ \ (\textup{by Lemma~\ref{weightls}(a)}) \nonumber \\
 &= \wt_v(3-\wt_v) \nonumber \\
 &\leq 0 \nonumber .
\end{align}
 
This implies that
\begin{align*}
 D''(v) - d(v) &= D'(v)-d(v)+2\# L_v \\
 &= 2 \big(l_v+s_v \big) -  \wt_v(\wt_v-1)  - \left( \sum_{\substack{v' \in C(v) \\ v' \ \textrm{even}}} \wt_{v'}(\wt_{v'}-1) \right) + 2\# L_v \\
 & \pushright{(\textup{by Equation~\ref{D'minusd}} )} \\
 &\leq  - \sum_{\substack{v' \in C(v) \\ v' \ \textrm{even}}} \wt_{v'}(\wt_{v'}-1) \ \ \ \ (\textup{by Equation~\ref{weight3}} )  \\
 &\leq 0 \ \ \ \ (\textup{by  Lemma  }\ref{weight}). 
\end{align*}

If $v$ is odd and $D''(v) = d(v)$, then either $\wt_v = 2$ or $\wt_v = 3$ and $r_v = 0$ and $\# L_v = s_v$. By Lemma~\ref{oddeven}, $r_v = 0$ if and only if $v$ has no even children. Since every child of $v$ has weight aleast $2$ and has weight bounded above by $\wt_v = 3$, Lemma~\ref{oddeven} tells us that $\#L_v = s_v$.
\end{proof}


We are now ready to prove the main theorem.
\begin{proof}[Proof of Theorem~\ref{main}]
Construct the proper regular model $X$ as above. Let $n(X)$ denote the number of irreducible components of the special fiber of $X$ and let $n$ be the number of components of the special fiber of the minimal proper regular model $\mathcal{X}$ of $C$ . 

To prove $-\Art(X/S) \leq \nu(\Delta)$, sum the inequality of Lemma~\ref{localinequality} over all vertices of $T_B$ and use Lemmas~\ref{newsum} ~\ref{newnewsum}.

We have the equalities 
\[ -\Art(X/S) = n(X)-1+\tilde{f} \]
\[ -\Art(\mathcal{X}/S) = n-1+\tilde{f} \]
where $\tilde{f}$ is the conductor of the $\ell$-adic representation $\Gal (\overline{K}/K) \rightarrow \Aut_{\Q_\ell} \ (H^1_{\mathrm{et}}(X_{\overline{\eta}}, \Q_\ell))$ \cite[p.53, Proposition 1]{liup}. 
The minimal proper regular model can be obtained by blowing down some subset (possibly empty) of irreducible components of the special fiber of $X_s$, so $n \leq n(X)$. 

Putting everything together, we get
\[ -\Art(\mathcal{X}/S) \leq -\Art(X/S) \leq \nu(\Delta) .\]
 
\end{proof}

\begin{remark}
 \textup{Lemma~\ref{localinequality} and the proof of Theorem~\ref{main} tell us that $-\Art(\mathcal{X}/S) = \nu(\Delta)$ if and only if the model $X$ is already minimal and the tree $T_B$ satisfies certain strict conditions. Call a subset $S$ of vertices of $T_B$ a connecting chain if 
 \begin{itemize}
  \item for any $v \in V(T_B)$, if $v$ lies in the path between two vertices of $S$, then $v \in S$, and,
  \item every vertex in $S$ has exactly two neighbours in $T_B$.
 \end{itemize}                                                                                                                                                                                                                                                                              
 If $-\Art(X/S) = \nu(\Delta)$, then the conditions on the tree $T_B$ tell us that if we replace every connecting chain of $3$ or more vertices with a chain of $2$ vertices (or equivalently, disregard the length of the chains in $T_B$ and just consider the underlying topological space of $T_B$), then the tree $T_B$ has height at most $2$ (that is, the path from any vertex to the root has at most one other vertex), and all children of the root have at most $3$ neighbours. The model $X$ is not minimal if and only if it has contractible $-1$ curves, and this happens if and only if the tree $T_B$ has an odd vertex $v$ such that $l'_v = 0$, $v$ has an even parent, and $v$ has exactly one child, and that child is even.}
\end{remark}

\begin{corollary}\label{number}
Let $n$ be the number of components of the special fiber of the minimal proper regular model of $C$ over $R$. Then,
 \[ n \leq \nu(\Delta) + 1 .\] 
\end{corollary}
\begin{proof}
Since the conductor $\tilde{f}$ is a nonnegative integer, $n-1 \leq n-1+\tilde{f} \leq \nu(\Delta)$.
\end{proof}

\begin{remark}
 \textup{The equality $n=\nu(\Delta)+1$ holds if and only if $\tilde{f} = 0$ in addition to all the conditions for $-\Art(\mathcal{X}/S) = \nu(\Delta)$ to hold. By the N\'{e}ron-Ogg-Shafarevich criterion, $\tilde{f} = 0$ if and only if the Jacobian of $C$ has good reduction.}
\end{remark}

\section*{Acknowledgements}

I would like to thank my advisor, Bjorn Poonen, for suggesting this problem to me. I would also like to thank him for several helpful conversations, for his careful reading of the manuscript and for all the suggestions for improving the exposition.

\end{document}